\documentclass{amsart}
\usepackage[utf8]{inputenc}
\usepackage{microtype}
\usepackage{amsfonts, amsmath, amssymb, amsthm, mathtools}
\usepackage[foot]{amsaddr}

\usepackage{xcolor}
\definecolor{darkblue}{rgb}{0,0,0.6}
\usepackage{tikz}
\usetikzlibrary{matrix,arrows,trees}
\usepackage{tikz-cd}
\usepackage[colorlinks,colorlinks=true, citecolor=darkblue, filecolor=darkblue, linkcolor=darkblue, urlcolor=darkblue]{hyperref}
\usepackage[nameinlink, capitalise]{cleveref}
\hypersetup{bookmarksopen=true}

\Crefname{lem}{Lemma}{Lemmas}

\usepackage{eucal}                  

\numberwithin{equation}{section}

\swapnumbers
\newtheorem{theorem}[equation]{Theorem}
\newtheorem{prop}[equation]{Proposition}
\newtheorem{lem}[equation]{Lemma}
\newtheorem{cor}[equation]{Corollary}

\theoremstyle{definition}
\newtheorem{defn}[equation]{Definition}
\newtheorem{ex}[equation]{Example}

\newtheorem{rem}[equation]{Remark}

\renewcommand{\epsilon}{\varepsilon}
\renewcommand{\phi}{\varphi}

\newcommand{\bbZ}{\mathbb{Z}}

\newcommand{\cA}{\mathcal{A}}
\newcommand{\cB}{\mathcal{B}}
\newcommand{\cC}{\mathcal{C}}
\newcommand{\cD}{\mathcal{D}}
\newcommand{\cE}{\mathcal{E}}

\newcommand{\cK}{\mathcal{K}}

\newcommand{\cP}{\mathcal{P}}

\newcommand{\cU}{\mathcal{U}}

\newcommand*\cocolon{%
	\nobreak
	\mskip6mu plus1mu
	\mathpunct{}%
	\nonscript
	\mkern-\thinmuskip
	{:}%
	\mskip2mu
	\relax
}

\DeclareFontFamily{U}{min}{}
\DeclareFontShape{U}{min}{m}{n}{<-> dmjhira}{}
\newcommand{\yo}{\text{\usefont{U}{min}{m}{n}\symbol{'110}}}

\newcommand{\add}{\mathrm{add}}
\newcommand{\exct}{\mathrm{ex}}
\newcommand{\f}{\mathrm{f}}
\newcommand{\fg}{\mathrm{fg}}
\newcommand{\op}{\mathrm{op}}
\newcommand{\st}{\mathrm{st}}

\newcommand{\wt}{\mathrm{wt}}

\newcommand{\Ab}{\mathrm{Ab}}
\newcommand{\catadd}{\mathrm{Cat}^\add}
\newcommand{\catrex}{\mathrm{Cat}^\mathrm{rex}}
\newcommand{\catst}{\mathrm{Cat}^\st}
\newcommand{\catheart}{\mathrm{Cat}^\heartsuit}
\newcommand{\catweight}{\mathrm{Cat}^\wt}
\newcommand{\derb}{D^\mathrm{b}}
\newcommand{\exact}{\mathrm{Exact}}
\newcommand{\eff}{\mathrm{eff}}
\newcommand{\pshadd}{\cP_\Sigma}
\newcommand{\Ac}{\mathrm{Ac}}
\newcommand{\An}{\mathrm{An}}
\newcommand{\ing}{\mathrm{in}}
\newcommand{\pr}{\mathrm{pr}}
\newcommand{\stabadd}{\mathrm{St}^{\add}}
\newcommand{\stab}{\mathrm{St}}
\newcommand{\Sp}{\mathrm{Sp}}

\DeclareMathOperator{\cofib}{cofib}
\DeclareMathOperator{\coker}{coker}
\DeclareMathOperator{\colim}{colim}
\DeclareMathOperator{\fgt}{fgt}
\DeclareMathOperator{\fib}{fib}
\DeclareMathOperator{\Fun}{Fun}
\DeclareMathOperator{\Hom}{Hom}
\DeclareMathOperator{\id}{id}
\DeclareMathOperator{\inc}{inc}

\DeclareMathOperator{\K}{K}
\DeclareMathOperator{\Nat}{Nat}
\DeclareMathOperator{\RHom}{RHom}
\DeclareMathOperator{\SW}{SW}

\DeclareRobustCommand{\SkipTocEntry}[5]{}

\title{On exact categories and their stable envelopes}
\thanks{This version of the article has been accepted for publication, after peer review, but is not the Version of Record and does not reflect post-acceptance improvements, or any corrections. The Version of Record is available online at: \url{http://dx.doi.org/10.1007/s00209-025-03904-6}}

\author{Victor Saunier}
\address[VS]{Universit\'e Sorbonne Paris Nord, LAGA, CNRS, 93430 Villetaneuse, France \\ Fakult\"at f\"ur Mathematik, Universit\"at Bielefeld, 33615 Bielefeld, Germany}
\email{saunier@math.univ-paris13.fr, vsaunier@math.uni-bielefeld.de}

\author{Christoph Winges}
\address[CW]{Fakult\"at f\"ur Mathematik, Universit\"at Regensburg, 93040 Regensburg, Germany}
\email{christoph.winges@ur.de}

\date{}

\begin{document}

\begin{abstract}
    We show that Klemenc's stable envelope of exact $\infty$-categories induces an equivalence between stable $\infty$-categories with a bounded heart structure and weakly idempotent complete exact $\infty$-categories. 
    Moreover, we generalise the Gillet--Waldhausen theorem to the connective algebraic K-theory of exact $\infty$-categories and deduce a universal property of connective algebraic K-theory as an additive invariant on exact $\infty$-categories.
    
    A key tool is a generalisation of a theorem due to Keller which provides a sufficient condition for an exact functor to induce a fully faithful functor on stable envelopes.
\end{abstract}


\maketitle
\tableofcontents


\section{Introduction}

Exact categories were invented by Quillen in \cite{quillen:higher-k} to get around the following problem: if $R$ is a regular ring, then there exist two equivalent presentations of $K_0(R)$,
either as the group-completion of the set of isomorphism classes of objects in the additive category of finitely generated projective modules, or by considering the category of finitely presented modules and further modding out by exact sequences. 
Exact categories provide an interpolation between these two extremes.

Barwick observed in \cite{barwick:heart} that a suitable formulation of the definition of an exact category can be adopted verbatim to define the notion of an exact $\infty$-category as an additive $\infty$-category equipped with a notion of exact sequences, see \cref{sec:basics}.
Small exact $\infty$-categories and exact functors between them assemble into an $\infty$-category $\exact$.

Minimally, every additive $\infty$-category becomes an exact $\infty$-category by declaring only the split bifibre sequences to be exact.
Maximally, a stable $\infty$-category carries an exact structure in which every bifibre sequence is an exact sequence.
It is easy to check that the $\infty$-categories $\catadd$ of additive $\infty$-categories and $\catst$ of stable $\infty$-categories both become full subcategories of $\exact$ in this way.
The inclusion of $\catadd$ into $\exact$ admits a right adjoint by taking the underlying additive $\infty$-category of an exact $\infty$-category,
and we will recall in \cref{sec:gabriel-quillen} results of Klemenc which show that the inclusion of $\catst$ admits a left adjoint
\[ \stab \colon \exact\to\catst \]
called the \textit{stable envelope}.

For ordinary exact categories $\cE$, the study of the functor $\cE \to \derb(\cE)$ to the bounded derived category has received plenty of attention; here, we understand $\derb(\cE)$ as the localisation of the category of bounded chain complexes in $\cE$ at the quasi-isomorphisms.
Keller in \cite{keller:derived-cats-univ-prob} and subsequently Porta in \cite{porta:up-triangulated-derivators} have provided universal properties for this map when considering the homotopy category of $\derb(\cE)$ and its triangulated structure, and refining it via the theory of derivators.
If $\cE$ is abelian, then Antieau, Gepner and Heller have given in \cite{agh:k-obstr-to-t-str} a universal property of the bounded derived category as a stable category equipped with a t-structure.
For an arbitrary exact category $\cE$, Bunke, Cisinski, Kasprowski and the second author have given in \cite{bckw:controlled} a universal property of the stable $\infty$-category $\derb(\cE)$, thus showing that it coincides with the stable envelope $\stab(\cE)$.

Returning to the general case of exact $\infty$-categories $\cE$,
Klemenc has shown in \cite{klemenc:stablehull} that the canonical map $\cE\to\stab(\cE)$ is fully faithful, closed under extensions and detects exact sequences. 
In particular, combined with the above, this is a version of the Gabriel--Quillen embedding theorem in the context of higher categories.
See \cref{sec:gabriel-quillen} for more details.

The main goal of this paper is to study further properties of the adjunction $\stab \colon \exact \rightleftarrows \catst \cocolon \inc$, and to discuss some consequences of these results for algebraic K-theory.\\

We can recast the fully faithfulness of $\cE \to \stab(\cE)$ in the light of a more general phenomenon. The category $\Ab^\fg$ of finitely generated abelian groups is a full subcategory of $\Sp^\omega$, the $\infty$-category of compact spectra, and it is closed under extensions.
Passage to stable envelopes induces a map
\[ \stab(\Ab^\fg) \simeq \derb(\bbZ) \to\Sp. \]
It is well-known that this map is not fully faithful.
However, as Klemenc's result shows, there are many cases where a full subcategory closed under extensions does induce a fully faithful map (in fact an equivalence) upon passage to stable envelopes.
Our first result identifies a property which guarantees that taking stable envelopes preserves fully faithfulness.

\begin{defn}
    Let $\cE$ be an exact $\infty$-category and let $\cU$ be a full subcategory closed under extensions.
    Then $\cU$ is \emph{left special} in $\cE$ if for every projection $p \colon x\twoheadrightarrow v$ with $x \in \cE$ and $v \in \cU$, there exists a morphism $f \colon u \to x$ with $u \in \cU$ such that the composite $p \circ f \colon u \to v$ is a projection in $\cU$.
\end{defn}

Left special subcategories have received a certain amount of attention in the study of exact categories.
We have borrowed the term from \cite{schlichting:delooping}, but the condition itself goes back to at least \cite{keller:derived-cats}.
In fact, \cref{thm:keller} establishes the following generalisation of \cite[Theorem~12.1]{keller:derived-cats}.

\begin{theorem}[Keller's criterion]
 Let $\cE$ be an exact $\infty$-category. 
 If $\cU \subseteq \cE$ is a left special subcategory of $\cE$, then the induced functor $\stab(\cU) \to \stab(\cE)$ is fully faithful.
\end{theorem}

This also holds for right special inclusions, i.e.\ those inclusions whose opposite is left special. Since the composite of a left and a right special inclusion need not be special on either side, this provides many other examples. 
An alternative proof of \cref{thm:keller} can be found in \cite{nw:presentable-stable-envelope}.

For a general exact $\infty$-category $\cE$, the construction of $\stab(\cE)$ proceeds as follows.
Consider $\stabadd(\cE)$, the stabilisation of the underlying additive $\infty$-category of $\cE$. 
This admits a tractable model as the Spanier--Whitehead stabilisation of $\pshadd^\f(\cE)$, which is by definition the smallest full subcategory of $\pshadd(\cE) := \Fun^\times(\cE^\op, \Sp_{\geq 0})$ containing the image of the (enriched) Yoneda embedding $\yo \colon \cE \to \pshadd(\cE)$ and closed under finite colimits.
In particular, $\stabadd(\cE)$ is a full subcategory of $\Fun^\times(\cE^\op, \Sp)$.

Then, to build $\stab(\cE)$, it suffices to force the image of the non-split exact sequences of $\cE$ to be exact. 
This can be achieved by considering for each exact inclusion $i \colon x \rightarrowtail y$ with cofibre $z$ the canonical comparison map
\[ c(i) \colon \cofib(\yo(i)) \to \yo(z). \]
Denote the cofibre of $c(i)$ by $E(i)$.
Defining $\Ac(\cE)$, the category of \textit{acyclics} over $\cE$, to be the stable subcategory of $\stabadd(\cE)$ generated by objects of the form $E(i)$,
the stabilisation of $\cE$ can be defined as the Verdier quotient
\[ \stab(\cE) := \stabadd(\cE)/\Ac(\cE) \]
of $\stabadd(\cE)$ by the subcategory of acyclics.

In \cite{neeman:counterexamples}, Neeman shows that if $\cE$ is an idempotent-complete exact category, then $\Ac(\cE)$ admits a bounded t-structure.
He also identifies the heart as the category of effaceable functors $\cE^\op\to\Ab$, as defined by Schlichting in Lemma 9 of Section 11 of \cite{schlichting:negative-kth-derived-cat}.

It is shown in \cite{rsw:every-spectrum} that this phenomenon holds in the higher world when $\cE$ is the underlying additive category of a stable $\infty$-category and \cite{klemenc:stablehull} hints at this phenomenon more generally for every exact $\infty$-category $\cE$.
\cref{prop:elementary-acyclics} and \cref{cor:acyclics-t-structure} show that this is indeed the case.

\begin{prop}
    For every exact $\infty$-category $\cE$, the stable $\infty$-category $\Ac(\cE)$ carries a bounded t-structure whose heart $\eff(\cE)$ is the full subcategory consisting precisely of objects of the form $E(i)$ for some exact inclusion $i$.
    In particular, all objects in the heart are effaceable in Schlichting's sense.
\end{prop}

This statement has also been observed by Efimov \cite[Proposition~G.7]{efimov:continuous-k}.

Since $\pshadd^\f(\cE)$ is a full subcategory of $\stabadd(\cE)$, it has an essential image in $\stab(\cE)$, which we denote $\stab_{\geq 0}(\cE)$. Adapting Klemenc's arguments, we are able to show in \cref{thm:gabriel-quillen} the following refinement of the Gabriel--Quillen embedding; see \cref{sec:basics} for terminology.

\begin{theorem}
 Let $\cE$ be an exact $\infty$-category.
 Then $j \colon \cE \to \stab_{\geq 0}(\cE)$ exhibits $\cE$ as an extension-closed full subcategory, and the induced exact structure on $\cE$ coincides with the given exact structure.
 In addition, the following are equivalent:
 \begin{enumerate}
     \item $\cE$ is weakly idempotent complete;
     \item $\cE$ is closed under fibres of projections in $\stab_{\geq 0}(\cE)$.
 \end{enumerate}
\end{theorem}

By a theorem of Sosnilo \cite{sosnilo:heart}, see also \cref{prop:weight-structure} below, the stable $\infty$-category $\stabadd(\cE)$ carries a bounded weight structure as it is the stable envelope of the underlying additive $\infty$-category of $\cE$ --- in particular, the \textit{weight heart} of this structure is the weak idempotent completion of the underlying additive $\infty$-category of $\cE$.
In addition, taking weight hearts induces an equivalence between stable $\infty$-categories with bounded weight structures and weakly idempotent complete additive $\infty$-categories.
Since $\stab(\cE)$ is a Verdier quotient of $\stabadd(\cE)$, it also inherits a structure that we now describe:

\begin{defn}
 Let $\cC$ be a stable $\infty$-category.
 A \emph{heart structure} on $\cC$ is a pair of full subcategories $(\cC_{\geq 0},\cC_{\leq 0})$ satisfying the following properties:
 \begin{enumerate}
     \item $\cC_{\geq 0}$ is closed under extensions and finite colimits, and $\cC_{\leq 0}$ is closed under extensions and finite limits;
     \item for every $x \in \cC$, there exists a cofibre sequence
     \[ x_{\leq 0} \to x \to \Sigma x_{\geq 0} \]
     with $x_{\leq 0} \in \cC_{\leq 0}$ and $x_{\geq 0} \in \cC_{\geq 0}$.
 \end{enumerate}
 Denote $\cC_{\geq n} := \Sigma^n\cC_{\geq 0}$ and $\cC_{\leq n} := \Sigma^n\cC_{\leq 0}$.
 Then the heart structure $(\cC_{\geq 0},\cC_{\leq 0})$ is \emph{bounded above} if $\cC = \bigcup_{n \in \bbZ} \cC_{\leq n}$, and it is \emph{bounded below} if $\cC = \bigcup_{n \in \bbZ} \cC_{\geq n}$.
 It is \emph{bounded} if it is both bounded above and bounded below.

 Let $\cD$ be another stable $\infty$-category with heart structure $(\cD_{\geq 0},\cD_{\leq 0})$.
 A \emph{heart-exact} functor $f \colon \cC \to \cD$ is an exact functor satisfying $f(\cC_{\geq 0}) \subseteq \cD_{\geq 0}$ and $f(\cC_{\leq 0}) \subseteq \cD_{\leq 0}$.

 Define the $\infty$-category of bounded heart categories $\catheart$ as the $\infty$-category of small stable $\infty$-categories with a bounded heart structure and heart-exact functors between them.
\end{defn}

\begin{rem}\label[rem]{rem:weight-structures}
    To illustrate the distinction, recall from \cite{bondarko:weight-structures,pauksztello:co-t,hs:moduli-spaces-of-hermitian-forms} that a \emph{weight structure} on a stable $\infty$-category $\cC$ is a pair of full subcategories $(\cC_{\geq 0},\cC_{\leq 0})$ satisfying the following properties:
    \begin{enumerate}
        \item $\cC_{\geq 0}$ and $\cC_{\leq 0}$ are closed under retracts;
        \item for every $x \in \cC$, there exists a cofibre sequence
     \[ x_{\leq 0} \to x \to \Sigma x_{\geq 0} \]
     with $x_{\leq 0} \in \cC_{\leq 0}$ and $x_{\geq 0} \in \cC_{\geq 0}$;
        \item\label{it:weight-3} for every $x \in \cC_{\leq 0}$ and every $y \in \cC_{\geq 0}$, the mapping spectrum $\hom_\cC(x,y)$ is connective.
    \end{enumerate}
    By \cite[Lemma~3.1.2]{hs:moduli-spaces-of-hermitian-forms}, the subcategory $\cC_{\geq 0}$ is closed under extensions and finite colimits and $\cC_{\leq 0}$ is closed under extensions and finite limits. So any weight structure is a heart structure.
    Moreover, \cite[Lemma~3.1.2]{hs:moduli-spaces-of-hermitian-forms} shows that either part of a weight structure determines the other part through the orthogonality condition \eqref{it:weight-3}.
\end{rem}

Heart categories were introduced by the first author in \cite[Definition~2.1]{saunier:heart}, where it was shown that if $\cC$ is a bounded heart category, then $\stab(\cC^\heartsuit)\to\cC$ is an equivalence.
As we will show in \cref{sec:bounded-heart-categories},
this result can also be recovered via Keller's criterion.
In \textit{loc.~cit.}, the generalisation of Sosnilo's result for heart structures was not touched upon.
Using the above material, we can give the following answer:

\begin{theorem}\label[theorem]{thm:main}
    Let $\cE$ be an exact $\infty$-category. 
    Then there is a bounded heart structure on $\stab(\cE)$ whose heart is the weak-idempotent completion of $\cE$.
    
    In consequence, there exists an adjunction
    \[ \stab \colon \exact \rightleftarrows \catheart \cocolon (-)^\heartsuit \]
    which identifies $\catheart$ with the full subcategory of weakly idempotent complete exact $\infty$-categories.

    Moreover, this adjunction restricts to the adjunction
    \[ \stabadd \colon \catadd \rightleftarrows \catweight \cocolon (-)^\heartsuit \]
    obtained by Sosnilo in \cite[Corollary~3.4]{sosnilo:heart},
    which identifies the $\infty$-category of stable $\infty$-categories with bounded weight structure with the full subcategory of weakly idempotent complete additive $\infty$-categories.
\end{theorem}

The following statement about the interaction between algebraic K-theory and stable envelopes can be deduced from \cite[Theorem~3.1]{saunier:heart} by following Quillen's original arguments for the resolution theorem of \cite{quillen:higher-k}.
We will provide a short argument in \cref{sec:bounded-heart-categories} explaining how to assemble the proof.

\begin{theorem}[Gillet--Waldhausen]\label[theorem]{thm:gillet-waldhausen}
 	Let $\cE$ be an exact $\infty$-category. Then
 	\[ \K(j) \colon \K(\cE) \to \K(\stab(\cE)) \]
 	is an equivalence, where $\K$ denotes connective algebraic K-theory.
\end{theorem}

Since $\stab(\cE)$ actually is the bounded derived category in the case of an exact category $\cE$ by \cite[Corollary~7.4.12]{bckw:controlled}, this is a generalisation of \cite[Theorem~1.11.7]{TT}, justifying its name.

In fact, the above discussion allows us to prove a version of the universal characterisation of algebraic K-theory proven by Blumberg, Gepner and Tabuada \cite{bgt:univcharkt}.
Namely, the notion of semi-orthogonal decompositions admits a straightforward generalisation to exact $\infty$-categories, and we say that an invariant is additive if it sends those decompositions to (split-)exact sequences of a stable category.

\begin{theorem}
    The natural transformation $\Sigma^\infty\iota\to\K$ of functors $\exact\to\Sp$ is initial amongst natural transformations $\Sigma^\infty\iota\to F$ with target a functor $F$ which is additive and invariant under passage to the stable envelope.
\end{theorem}

Note that by \cite[Theorem~3.3]{saunier:heart}, being invariant under passage to the stable envelope is equivalent to
satisfying the resolution theorem (see \cite[Theorem~3]{quillen:higher-k} and \cite[Theorem~1.4]{saunier:heart})
and being invariant under weak-idempotent completion.
In particular, theorems 2 and 3 of Quillen's seminal paper \cite{quillen:higher-k} identify properties of the K-theory functor which allow for the formulation of a universal property.

Finally, let us mention that Keller's criterion implies that $\stab$ sends semi-orthogonal decompositions of exact $\infty$-categories to semi-orthogonal decompositions of stable $\infty$-categories.
In particular, it follows that the additive non-commutative motives of Blumberg, Gepner and Tabuada are a full subcategory of a similar construction for exact $\infty$-categories.

\addtocontents{toc}{\SkipTocEntry}
\subsection*{Conventions}
Throughout the body of this article, we adopt the following conventions:
\begin{itemize}
    \item From this point onwards, the word ``category'' means ``$\infty$-category''.
    \item We denote by $\An$ the category of anima/$\infty$-groupoids/spaces/weak homotopy types.
    \item  If we want to emphasise that a category has discrete mapping anima, we will refer to it as a ``1-category'' or ``ordinary category''.
    \item For an arbitrary category $\cC$, we denote by $\Hom_\cC$ the mapping anima of $\cC$.
    If $\cC$ is stable, we use $\hom_\cC$ to denote the mapping spectra of $\cC$.
    \item If $\cC$ is a category with finite products, we denote by $\Fun^\times(\cC,-)$ the full subcategory of functors preserving finite products.
    If $\cC$ is additive, we also use the notation $\Fun^\oplus(\cC,-)$.
    If $\cC$ and $\cD$ are stable categories, $\Fun^\exct(\cC,\cD)$ denotes the category of exact functors from $\cC$ to $\cD$.
    \item If $\cA$ and $\cB$ are wide subcategories of a category $\cC$, we say that \emph{$\cA$ is closed under pushouts along $\cB$} if every span $y \xleftarrow{f} x \xrightarrow{g} z$ with $f$ in $\cA$ and $g$ in $\cB$ can be completed to a pushout square, and for any such pushout square the induced morphism $z \to y \cup_x z$ also lies in $\cA$, and similarly for pullbacks.
\end{itemize}

\addtocontents{toc}{\SkipTocEntry}
\subsection*{Acknowledgements}
We thank Markus Land and Vova Sosnilo for comments on earlier versions of this article.
Moreover, we are grateful to the referees for their thorough report which helped to improve the manuscript in numerous places.

VS was supported by the European Research Council as part of the project \textit{Foundations of Motivic Real K-Theory} (ERC grant no. 949583).
CW was supported by the CRC 1085 ``Higher Invariants" funded by the Deutsche Forschungsgemeinschaft (DFG).


\section{Exact categories}\label[section]{sec:basics}

Let us recall the basic definitions concerning exact categories, see \cite[Section~3]{barwick:heart}.

\begin{defn}
    An \emph{exact category} is an additive category $\cE$ together with a choice of two wide subcategories $\ing\cE$ and $\pr\cE$, the \emph{(exact) inclusions} and \emph{(exact) projections}, such that the following properties are satisfied:
    \begin{enumerate}
        \item for every object $x \in \cE$, the unique morphism $0 \to x$ is an inclusion and the unique morphism $x \to 0$ is a projection;
        \item the class of inclusions is stable under pushouts along arbitrary maps, and the class of projections is stable under pullback along arbitrary maps;
        \item for a commutative square
        \[\begin{tikzcd}
            x\ar[r, "i"]\ar[d,"p"'] & y\ar[d, "q"] \\ x'\ar[r, "j"] & y'
        \end{tikzcd}\]
        the following are equivalent:
        \begin{enumerate}
            \item the morphism $i$ is an inclusion, the morphism $p$ is a projection, and the square is a pushout;
            \item the morphism $j$ is an inclusion, the morphism $q$ is a projection, and the square is a pullback;
        \end{enumerate}
        Any square of this form is called an \emph{exact square}.
        If $x' \simeq 0$, we also say that $x \xrightarrow{i} y \xrightarrow{q} y'$ is an \emph{exact sequence}.
    \end{enumerate}
    We will typically denote inclusions by feathered arrows $\rightarrowtail$ and projections by two-headed arrows $\twoheadrightarrow$.
    
   Let $(\cE',\ing\cE',\pr\cE')$ be another exact category.
   An \emph{exact functor} $f \colon \cE \to \cE'$ is a reduced functor which sends exact squares to exact squares.
   We denote by $\exact$ the category of small exact categories and exact functors.
   Moreover, we use $\Fun^\exct(\cE,\cE')$ to denote the category of exact functors from $\cE$ to $\cE'$.
\end{defn}

\begin{rem}\ 
\begin{enumerate}
    \item We will typically suppress the chosen subcategories $\ing\cE$ and $\pr\cE$ from notation.
    \item An exact functor in particular sends inclusions to inclusions and projections to projections.
\end{enumerate}
\end{rem}

There is a number of standard diagram lemmas that make it convenient to work with exact categories.
The most central one is the following.

\begin{lem}[{\cite[Proposition~A.1]{klemenc:stablehull} and \cite[Lemma~4.7]{barwick:heart}}]\label[lem]{lem:pushouts}
    Let $\cE$ be an exact category and consider a commutative square
    \[\begin{tikzcd}
        x\ar[r, rightarrowtail, "i"]\ar[d] & y\ar[d] \\
        x'\ar[r, rightarrowtail, "j"]& y'
    \end{tikzcd}\]
    in which $i$ and $j$ are inclusions as indicated.
    Then the following are equivalent:
    \begin{enumerate}
        \item the square is bicartesian;
        \item the square is a pushout;
        \item the induced map $\cofib(i) \to \cofib(j)$ on cofibres is an equivalence.
    \end{enumerate}
    Moreover, in this case there exists an exact sequence $x \rightarrowtail x' \oplus y \twoheadrightarrow y'$.
\end{lem}

\cref{lem:pushouts} implies that a functor between exact categories is exact if and only if it preserves exact sequences.
This in turn is equivalent to requiring that the functor is an exact functor between the underlying Waldhausen categories \cite[Proposition~4.8]{barwick:heart}.

Moreover, a version of the five lemma holds for inclusions and projections in an exact category.

\begin{lem}\label[lem]{lem:five-lemma}
 Let $\cE$ be an exact category and consider a commutative diagram
 \[\begin{tikzcd}
     x\ar[r,tail]\ar[d, tail, "f"] & y\ar[r, two heads]\ar[d, "g"] & z\ar[d, "h"] \\
     x'\ar[r, tail] & y'\ar[r, two heads] & z'
 \end{tikzcd}\]
 with both rows exact. Assume that $f$ is an inclusion.
 Then $h$ is an inclusion if and only if the induced map $x' \cup_x y \to y'$ is an inclusion.
 If this is the case, $g$ is also an inclusion.
\end{lem}
\begin{proof}
 The proofs of \cite[Proposition~3.1 \& Corollary~3.2]{buehler:exact-cats} can be copied almost verbatim.
\end{proof}

Another easy consequence is Quillen's ``obscure axiom''.

\begin{lem}\label[lem]{lem:obscure-axiom}
    Let $\cE$ be an exact category.
    Suppose that $p \colon x \to y$ is a morphism in $\cE$ which admits a fibre and that there exists a morphism $f \colon x' \to x$ such that $pf$ is a projection in $\cE$.
    Then $p$ is a projection.
\end{lem}
\begin{proof}
    The proof of \cite[Proposition~2.16]{buehler:exact-cats} can be copied almost verbatim.
\end{proof}

It will often be convenient to assume the following additional property on the underlying additive category of an exact category.

\begin{defn}
    An additive category $\cA$ is \emph{weakly idempotent complete} if for every retraction diagram $i \colon a \rightleftarrows b \cocolon r$ there exists an object $a' \in \cA$ and a morphism $s \colon b \to a'$ such that $r + s \colon b \to a \oplus a'$ is an equivalence.
\end{defn}

\begin{ex}
\ \begin{enumerate}
 \item Every idempotent complete additive category is weakly idempotent complete.
 \item Every stable category is weakly idempotent complete.
 \item For an ordinary ring $R$, the additive $1$-category of finitely generated, stably free $R$-modules is weakly idempotent complete.
 If there exists a stably free $R$-module which is not free, it follows that the category of finitely generated free $R$-modules is not weakly idempotent complete.

 Since the idempotent completion of both of these categories is the category of finitely generated projective $R$-modules, the category of finitely generated, stably free $R$-modules is idempotent complete if and only if every finitely generated projective $R$-module is stably free. This is equivalent to the statement that $\widetilde{K}_0(R) := \coker(K_0(\bbZ) \to K_0(R))$ is trivial.
 \item The heart of a heart category $(\cC,\cC_{\geq 0},\cC_{\leq 0})$ is weakly idempotent complete: every retraction diagram in $\cC^\heartsuit$ admits a complement in $\cC$.
 The resulting complement is both a fibre of a morphism in $\cC_{\leq 0}$ and a cofibre of a morphism in $\cC_{\geq 0}$, so it lies in $\cC^\heartsuit$ as well.
\end{enumerate}
\end{ex}

The following notion provides an easy way to build new exact categories from a given exact category.

\begin{defn}
    Let $\cE$ be an exact category and let $\cU$ be a full subcategory.
    Then $\cU$ is \emph{extension-closed} in $\cE$ if it contains a zero object and for every exact sequence $u \rightarrowtail y \twoheadrightarrow w$ with $u, w \in \cU$, we also have $y \in \cU$.
\end{defn}

\begin{lem}\label[lem]{lem:ext-closed-subcat}
    If $\cU$ is an extension-closed subcategory of the exact category $\cE$, then
    \[ \ing\cU := \{ i \in \cU \cap \ing\cE \mid \cofib(i) \in \cU \}\quad\text{and}\quad \pr\cU := \{ p \in \cU \cap \pr\cE \mid \fib(p) \in \cU \} \]
    define the structure of an exact category on $\cU$.
\end{lem}
\begin{proof}
    Since $\cU$ is extension-closed, it is a full additive subcategory of $\cE$.
    For the remaining axioms, it suffices to show that $\cU$ is closed under pushouts along morphisms in $\ing\cU$ and under pullbacks along morphisms in $\pr\cU$.
    If $w \leftarrow u \xrightarrow{i} v$ is a span in $\cU$ with $i \in \ing\cU$, then the pushout $y$ fits into an exact sequence $w \rightarrowtail y \twoheadrightarrow \cofib(i)$ in $\cE$.
    Since $\cU$ is closed under extensions, it follows that $y \in \cU$.
    The case of pullbacks along morphisms in $\pr\cU$ is analogous.
\end{proof}

If we do not say anything else, extension-closed subcategories will always be equipped with the exact structure coming from \cref{lem:ext-closed-subcat}.

We will also be interested in the following property that an extension-closed subcategory may or may not satisfy.

\begin{defn}
    Let $\cE$ be an exact category and let $\cU$ be an extension-closed subcategory.
    Then $\cU$ is \emph{closed under fibres of projections} if for every exact sequence $x \rightarrowtail v \twoheadrightarrow w$ with $v,w \in \cU$, we also have $x \in \cU$.
\end{defn}


\section{\texorpdfstring{The Gabriel--Quillen embedding}{The Gabriel-Quillen embedding}}\label[section]{sec:gabriel-quillen}

Let us begin by providing more details of Klemenc's construction of the Gabriel--Quillen embedding.
If $\cA$ is additive, the functor
\[ \Omega^\infty \colon \Fun^\times(\cA^\op,\Sp_{\geq 0}) \to \Fun^\times(\cA^\op,\An) \]
is an equivalence by \cite[Corollary~4.9]{gepner-groth-nikolaus:multiplicative-loop-space-machines} and provides a lift
\[ \yo \colon \cA \to \Fun^\times(\cA^\op,\Sp_{\geq 0}) =: \pshadd(\cA) \]
of the Yoneda embedding to connective spectra.
Recall that $\pshadd^\f(\cA)$ denotes the category of \emph{finite additive presheaves}, the smallest full subcategory of $\pshadd(\cA)$ which contains the essential image of the Yoneda embedding and is closed under finite colimits.
It is prestable by \cite[Proposition~3.5]{klemenc:stablehull} and the upgraded Yoneda embedding $\yo \colon \cA \to \pshadd^\f(\cA)$ is the initial additive functor from $\cA$ with finitely cocomplete target \cite[Proposition~3.2]{klemenc:stablehull}.

Recall for example from \cite[Section~C.1.1]{SAG} that the fully faithful inclusion $\catst \to \catrex$ into the category of small, pointed categories with finite colimits and finite colimit-preserving functors admits a left adjoint
\[ \SW \colon \catrex \to \catst \]
called the \emph{Spanier--Whitehead stabilisation}.
If $\cC$ is a pointed category with finite colimits, its Spanier--Whitehead stabilisation is explicitly given by
\[ \SW(\cC) \simeq \colim \big( \cC \xrightarrow{\Sigma} \cC \xrightarrow{\Sigma} \cC \xrightarrow{\Sigma} \ldots \big). \]

\begin{defn}
    Define
    \[ \stabadd(\cA) := \SW(\pshadd^\f(\cA)) \]
    as the Spanier--Whitehead stabilisation of the category of finite additive presheaves. 
\end{defn}

By the universal property of $\SW$, the stable category $\stabadd(\cA)$ comes equipped with a functor $\cA \to \stabadd(\cA)$ such that the restriction functor
    \[ \Fun^\exct(\stabadd(\cA),\cC) \to \Fun^\oplus(\cA,\cC) \]
    is an equivalence for every stable category $\cC$.

In particular, there exists an adjunction
\[ \stabadd \colon \catadd \rightleftarrows \catst \cocolon \fgt \]
where $\fgt$ denotes the forgetful functor.
Since the inclusion functor $\Sp_{\geq 0} \to \Sp$ preserves colimits, we obtain an induced fully faithful functor
\[ \stabadd(\cA) \to \Fun^\times(\cA^\op,\Sp) \]
which identifies $\stabadd(\cA)$ with the smallest full stable subcategory of $\Fun^\times(\cA^\op,\Sp)$ containing each of the spectrum-valued presheaves $\yo(x) \colon \cA^\op \rightarrow \Sp_{\geq 0} \subseteq \Sp$ for $x \in \cA$.

An important property of $\stabadd(\cA)$ is the following.

\begin{prop}[Bondarko, Sosnilo]\label[prop]{prop:weight-structure}
    Let $\cA$ be a small additive category.
    Then $\stabadd_{\geq 0}(\cA) := \pshadd^\f(\cA) \subseteq \stabadd(\cA)$ is the positive part of a bounded weight structure on $\stabadd(\cA)$.
\end{prop}
\begin{proof}
    This is part of the proof of \cite[Corollary~3.4]{sosnilo:heart} and is discussed in more detail in \cite[Proposition~2.13]{saunier:heart}.
    For the reader's convenience, we also record a proof here.

    Since $\stabadd(\cA)$ is the Spanier--Whitehead stabilisation of $\stabadd_{\geq 0}(\cA)$,  every object becomes an object of $\stabadd_{\geq 0}(\cA)$ after a finite number of suspensions.
    Define
    \[ \stabadd_{\leq 0}(\cA) := \{ X \in \stabadd(\cA) \mid \hom_{\stabadd(\cA)}(X,Y) \text{ is connective for all } Y \in \stabadd_{\geq 0}(\cA) \}. \]
    Note that every representable presheaf is contained in $\stabadd_{\leq 0}(\cA)$.
    Moreover, if $X$ is the cofibre of a morphism between two objects in $\stabadd_{\leq 0}(\cA)$, then $\Sigma^{-1}X \in \stabadd_{\leq 0}(\cA)$.
    In particular, every object in $\stabadd_{\geq 0}(\cA)$ becomes an object of $\stabadd_{\leq 0}(\cA)$ after a finite number of desuspensions because the representable presheaves generate $\stabadd_{\geq 0}(\cA)$ under finite colimits.
    As every object is some iterated desuspension of an object in $\stabadd_{\geq 0}(\cA)$, the same assertion is true for every object in $\stabadd(\cA)$. 
    
    Since $\stabadd_{\geq 0}(\cA)$ is prestable, both $\stabadd_{\geq 0}(\cA)$ and $\stabadd_{\leq 0}(\cA)$ are closed under retracts.
    Since the orthogonality relation~\ref{rem:weight-structures}~\eqref{it:weight-3} holds by definition, we are left with checking that every object admits a weight decomposition.
    For this, it will be sufficient to show that for every $X \in \stabadd_{\geq 0}(\cA)$ and every $k \geq 0$, there exists a weight decomposition of degree $k$, by which we mean a cofibre sequence $X' \to X \to X''$ with $\Sigma^{-k}X' \in \stabadd_{\leq 0}(\cA)$ and $\Sigma^{-k-1}X'' \in \stabadd_{\geq 0}(\cA)$.
    
    Let $\cC$ be the full subcategory of $\stabadd(\cA)$ consisting of those objects which admit a $\pi_0$-surjection $\yo(x) \to X$ for some $x \in \cA$.
    Then every representable presheaf is contained in $\cC$.
    For any span $X_1 \leftarrow X_0 \to X_2$ of objects in $\cC$, the induced morphism $X_1 \oplus X_2 \to X_1 \cup_{X_0} X_2$ is a $\pi_0$-surjection, so $\cC$ is closed under finite colimits.
    
    In particular, $\stabadd_{\geq 0}(\cA) \subseteq \cC$.
    Consequently, for every object $X \in \stabadd_{\geq 0}(\cA)$ there exist some $x \in \cA$ and a cofibre sequence $\yo(x) \to X \to Y$ with $\Sigma^{-1}Y \in \stabadd_{\geq 0}(\cA)$.

    Assume that $X' \to X \to X''$ is a weight decomposition of degree $k$.
    Then there exists a $\pi_k$-surjection $\Sigma^k\yo(x'') \to \Sigma^{-1}X''$ for some $x'' \in \cA$.
    Letting $Y$ denote the cofibre of this morphism, we obtain another cofibre sequence
    \[ Y \to Y \cup_{\Sigma^{-1}X''} X' \to X. \]
    Since $\yo(x'')$ and $\Sigma^{-k}X'$ both lie in $\stabadd_{\leq 0}(\cA)$, it follows that $\Sigma^{-k-1}(Y \cup_{\Sigma^{-1}X''} X') \in \stabadd_{\leq 0}(\cA)$.
    By construction, $\Sigma^{-k-1}Y \in \stabadd_{\geq 0}(\cA)$, so rotating the above fibre sequence yields a weight decomposition of degree $k+1$.
    By induction, we conclude that every object in $\stabadd_{\geq 0}(\cA)$ admits a weight decomposition of each degree.
\end{proof}

\begin{rem}
    Using Bondarko's weight complex functor, which is by definition the exact functor $\stabadd(\cA) \to \cK^\mathrm{b}(\mathrm{ho}\cA)$ induced by the projection onto the homotopy category, one can check directly that the heart of the weight structure in \cref{prop:weight-structure} is a weak idempotent completion of $\cA$. This assertion also follows from the more general \cref{prop:stab-bounded-heart-structure} below.
\end{rem}

Assume now that $\cE$ is an exact category.
Forgetting the exact structure on $\cE$, it still makes sense to consider the additive stabilisation $\stabadd(\cE)$.
Recall from the introduction the definition of the object $E(i)$ associated to an exact inclusion $i$ in $\cE$.
Then an exact functor $\stabadd(\cE)\to\cC$ restricts to an exact functor $\cE \to \cC$ if and only if it maps each $E(i)$ to zero.

\begin{defn}\
\begin{enumerate}
    \item The category $\Ac(\cE)$ of \emph{acyclics} over $\cE$ is the smallest full stable subcategory of $\stabadd(\cE)$ which contains all objects of the from $E(i)$.
    \item Define $\stab(\cE)$, the \textit{stable envelope of $\cE$}, as the Verdier quotient
    \[ \stab(\cE) := \stabadd(\cE)/\Ac(\cE). \]
    By construction, there is an exact functor $j \colon \cE \to \stab(\cE)$.
\end{enumerate}
\end{defn}

The following can be assembled from \cite[Proposition~4.17]{klemenc:stablehull} and \cite[Proposition~4.22]{klemenc:stablehull}.

\begin{theorem}[Klemenc]\label[theorem]{thm:klemenc}
    Let $\cE$ be an exact category. Then $j \colon \cE \to \stab(\cE)$ is fully faithful and the restriction functor
    \[ j^* \colon \Fun^\exct(\stab(\cE),\cC) \to \Fun^\exct(\cE,\cC) \]
    is an equivalence for every stable category $\cC$.
    
    In particular,
    $\stab$ upgrades to a functor which takes part in an adjunction
    \[\begin{tikzcd}
        \exact\ar[r, shift left=1, "\stab"] & \catst\ar[l, shift left=1].
    \end{tikzcd}\]
    where the right adjoint is the fully faithful inclusion equipping a stable category with the maximal exact structure.
\end{theorem}

\begin{cor}\label[cor]{cor:stab-op}
     For every exact category $\cE$, the functor $\stab(\cE^\op) \to \stab(\cE)^\op$ induced by $j^\op$ is an equivalence.
\end{cor}
\begin{proof}
    This follows by unwinding the universal property supplied by \cref{thm:klemenc}:
    \begin{align*}
        \Fun^\exct(\stab(\cE^\op),\cC)
        &\simeq \Fun^\exct(\cE^\op,\cC) 
        \simeq \Fun^\exct(\cE, \cC^\op)^\op \\
        &\simeq \Fun^\exct(\stab(\cE),\cC^\op)^\op
        \simeq \Fun^\exct(\stab(\cE)^\op,\cC) \qedhere
    \end{align*}
\end{proof}

\begin{rem}
    Note that for every stable category $\cC$ (equipped with the maixmal exact structure) and every additive category $\cA$ (equipped with the split exact structure), we have $\Fun^\exct(\cC, \cA)\simeq 0$:
    such a functor $F \colon \cC\to\cA$ maps $X\to 0\to \Sigma X$ to a split-exact sequence, so $F(X)\simeq 0$.

    However, the orthogonal complement of $\catst$ in $\exact$ contains more than just additive categories.
    Via the embedding $\cE \to \stab(\cE)$, we obtain a notion of mapping spectra in $\cE$.
    If $\cE$ has finite cohomological dimension in the sense that these mapping spectra
    are uniformly bounded below, then $\cE$ is also right-orthogonal to $\catst$:
    any object in the image of an exact functor $F \colon \cC \to \cE$ may be written as an arbitrarily high desuspension of another object, so we obtain for every $X \in \cC$, every $Y \in \cE$ and every natural number $l$ some object $X' \in \cC$ and an isomorphism
    \[ \pi_k\Hom_\cE(F(X),Y) \simeq \pi_0\Hom_{\stab(\cE)}(F(X'),\Sigma^{l-k}Y); \]
    the right hand term vanishes for $l$ sufficiently large by assumption.
\end{rem}

Resuming the above, we have seen that for every exact $\cE$, there is a diagram
\[
    \begin{tikzcd}
        & \cE^\oplus\arrow[r]\arrow[d] & \cE\arrow[d, "j"] \\
        \Ac(\cE)\arrow[r] & \stabadd(\cE)\arrow[r] & \stab(\cE)
    \end{tikzcd}
\]
Here, $\cE^\oplus$ denotes the split-exact category $\cE$, the bottom row is a Verdier sequence of stable categories and the vertical maps are exact and fully faithful functors.
In the following, we prove a refinement of Klemenc's result that $\cE \to \stab(\cE)$ exhibits $\cE$ as an extension-closed full subcategory \cite[Proposition~4.25]{klemenc:stablehull}.
First, we study the category of acyclics, which admits a fairly explicit description.

\begin{defn}
    An additive functor $F \colon \cE^\op \to \Ab$ is \emph{weakly effaceable} if for every $x \in \cE$ and $\alpha \in F(x)$ there exists a projection $p \colon x'\twoheadrightarrow x$ in $\cE$ with $p^*\alpha = 0$.
\end{defn}

\begin{prop}\label[prop]{prop:elementary-acyclics}
    Let $F \colon \cE^\op \to \Ab$ be an additive functor.
    Then the following are equivalent:
    \begin{enumerate}
        \item\label{it:elementary-acyclics-1} $F$ is equivalent to $E(i)$ for some inclusion $i$ of $\cE$;
        \item\label{it:elementary-acyclics-2} there exists a projection $p \colon y \twoheadrightarrow z$ in $\cE$ such that \[ F \simeq \coker( \pi_0\yo(p) \colon \pi_0\yo(y) \to \pi_0\yo(z) ); \]
        \item\label{it:elementary-acyclics-3} there exists a morphism $f \colon y \to z$ in $\cE$ such that
        \[ F \simeq \coker( \pi_0\yo(f) \colon \pi_0\yo(y) \to \pi_0\yo(z) ) \]
        and $F$ is weakly effaceable.
    \end{enumerate}
    Moreover, $E(i)$ is discrete for every inclusion $i$ of $\cE$.
\end{prop}
\begin{proof}
     Let $x \overset{i}{\rightarrowtail} y \overset{p}{\twoheadrightarrow} z$ be an exact sequence in $\cE$.
     Then $c(i)$ fits into the following commutative diagram in $\Fun^\times(\cE^\op,\Sp_{\geq 0})$:
     \[\begin{tikzcd}
      \yo(x)\ar[r, "\yo(i)"]\ar[d, "\id"] & \yo(y)\ar[r]\ar[d, "\id"] & \cofib(\yo(i))\ar[d, "c(i)"] \\
      \yo(x)\ar[r, "\yo(i)"] & \yo(y)\ar[r] & \yo(z)
     \end{tikzcd}\]
     The top row is a cofibre sequence while the bottom row is a fibre sequence, so $\cofib(c(i))$ is discrete.
     In particular, $\cofib(c(i)) \simeq \coker(\pi_0\yo(p))$.
     This proves the equivalence of \eqref{it:elementary-acyclics-1} and \eqref{it:elementary-acyclics-2}, and shows that $E(i)$ is always discrete.

     Assume \eqref{it:elementary-acyclics-2}.
     By the preceding discussion, $F(t)$ can be identified with the quotient of $\pi_0\Hom_\cE(t,z)$ by the subgroup of morphisms which factor over $p$.
     Given a morphism $f \colon t \to z$, form the pullback $t \times_z y$.
     Then the structure map $q \colon t \times_z y \to t$ is a projection, and the composition $t \circ q$ factors over $p$ by construction.
     Hence $F$ is weakly effaceable.

     For the converse, assume \eqref{it:elementary-acyclics-3}.
     Since $F$ is weakly effaceable, there exists a projection $p' \colon y' \twoheadrightarrow z$ such that the composite
     \[ \pi_0\yo(y')\xrightarrow{\pi_0\yo(p')} \pi_0\yo(z) \to F \]
is trivial.
Then $p := f + p' \colon y \oplus y' \to z$ is a projection by \cref{lem:five-lemma}, and $F \simeq \coker(\pi_0\yo(f + p'))$ since the images of $\pi_0\yo(f)$ and $\pi_0\yo(p)$ agree by construction.
\end{proof}

\begin{defn}
    Denote by $\eff(\cE) \subseteq \Fun^\times(\cE^\op,\Ab)$ the full subcategory of objects satisfying the equivalent conditions of \cref{prop:elementary-acyclics}.
\end{defn}

\begin{prop}\label[prop]{prop:acyclics}
    Identifying $\stabadd(\cE)$ with a full subcategory of $\Fun^\times(\cE^\op,\Sp)$, an object $A$ is acyclic if and only if $\pi_*A$ is concentrated in finitely many degrees and $\pi_kA$ lies in $\eff(\cE)$ for each $k$. 
\end{prop}
\begin{proof}
    Let $\cA$ denote the full subcategory of additive presheaves $A \colon \cE^\op \to \Sp$ such that $\pi_*A$ is concentrated in finitely many degrees and $\pi_kA$ lies in $\eff(\cE)$ for each $k$.
    Then $\cA$ contains all objects $E(i)$ by \cref{prop:elementary-acyclics}, and it is a stable subcategory by \cite[Corollary~3.13]{klemenc:stablehull}.
    Hence $\Ac(\cE) \subseteq \cA$.
    
    Conversely, an induction along the Postnikov tower shows that every object in $\cA$ is contained in the stable subcategory generated by $\eff(\cE)$.
\end{proof}

In particular, from the above proposition, we can deduce the following generalisation of \cite[Lemma~1.2]{neeman:counterexamples}.

\begin{cor}\label[cor]{cor:acyclics-t-structure}
    The Postnikov t-structure on $\Fun^\times(\cE^\op,\Sp)$ restricts to a bounded t-structure on $\Ac(\cE)$ whose heart is equivalent to $\eff(\cE)$.
\end{cor}
\begin{proof}
    This is immediate from \cref{prop:acyclics}.
\end{proof}

We will show in \cref{sec:bounded-heart-categories} that the image under the Verdier projection $\stabadd(\cE) \to \stab(\cE)$ of the bounded weight structure on $\stabadd(\cE)$ of \cref{prop:weight-structure} is the positive part of a bounded heart structure on $\stab(\cE)$.
As a first step in this direction, let us denote by $\stab_{\geq 0}(\cE)$ the essential image of $\stabadd_{\geq 0}(\cA)$ in $\stab(\cE)$. Then we observe the following.

\begin{lem}\label[lem]{lem:stab-positive-closure-properties}
 The full subcategory $\stab_{\geq 0}(\cE)$ is closed under finite colimits and extensions in $\stab(\cE)$.
\end{lem}
\begin{proof}
 This follows by combining Proposition~4.4, Proposition~4.7 and (the proof of) Corollary~4.8 in \cite{klemenc:stablehull}.
\end{proof}

\begin{lem}\label[lem]{lem:j-left-special}
    The embedding $\cE \to \stab_{\geq 0}(\cE)$ exhibits $\cE$ as a left special subcategory.
\end{lem}
\begin{proof}
    The following argument is contained in the proof of \cite[Proposition~4.25]{klemenc:stablehull}, but let us spell it out to highlight the key points.
    
    Consider an exact sequence $X \to Y \to j(z)$ in $\stab_{\geq 0}(\cE)$.
    The formula for mapping anima in the Verdier quotient $\stabadd(\cE)/\Ac(\cE)$ together with \cite[Proposition~4.6]{klemenc:stablehull} imply that the boundary map $j(z) \to \Sigma X$ can be represented by a zig-zag
    \[ \yo(z) \xrightarrow{d} G \xleftarrow{s} \Sigma F \]
    in $\cP_\Sigma^\f(\cE)$, where $F$ is a preimage of $X$ under the localisation functor and $\cofib(s)$ is acyclic.
    By \cref{prop:acyclics}, there exists a projection $p \colon y \twoheadrightarrow z$ in $\cE$ such that the composite map $\yo(y) \xrightarrow{d \circ \yo(p)} G \to \cofib(s)$ is nullhomotopic.
    Hence $d \circ \yo(p)$ lifts to $\Sigma F$.
    Since suspensions in $\cP_\Sigma^\f(\cE) \subseteq \Fun^\times(\cE^\op,\Sp_{\geq 0})$ are computed pointwise, $\Sigma F$ is $1$-connective. It follows that the zig-zag $\yo(y) \xrightarrow{d \circ \yo(p)} G \xleftarrow{s} \Sigma F$ represents the trivial map in $\stab_{\geq 0}(\cE)$, and consequently $j(p)$ lifts to $Y$.

    We are left with showing that $\cE$ is closed under extensions in $\stab_{\geq 0}(\cE)$. Given a cofibre sequence $j(x) \to Y \to j(z)$, the preceding discussion implies that there exists a commutative diagram
    \[\begin{tikzcd}
        j(x')\ar[r, "j(i)"]\ar[d] & j(y)\ar[r, "j(p)"]\ar[d] & j(z)\ar[d, "\id"] \\
        j(x)\ar[r] & Y\ar[r] & j(z)
    \end{tikzcd}\]
    in which both rows are cofibre sequences.
    \cref{lem:pushouts} implies that the left square is a pushout in $\stab_{\geq 0}(\cE)$.
    Since $i$ is an inclusion and $j$ preserves pushouts along inclusions, it follows that $Y \simeq j(x \cup_{x'} y)$.
\end{proof}

The following theorem generalises the classical Gabriel--Quillen embedding \cite[Appendix~A]{TT}.

\begin{theorem}[{Gabriel--Quillen embedding theorem \cite[Theorem~1.2]{klemenc:stablehull}}]\label[theorem]{thm:gabriel-quillen}
 Let $\cE$ be an exact category.
 Then $j \colon \cE \to \stab_{\geq 0}(\cE)$ exhibits $\cE$ as an extension-closed full subcategory, and the induced exact structure on $\cE$ coincides with the given exact structure.
 In addition, the following are equivalent:
 \begin{enumerate}
     \item $\cE$ is weakly idempotent complete;
     \item $\cE$ is closed under fibres of projections in $\stab_{\geq 0}(\cE)$.
 \end{enumerate}
\end{theorem}
\begin{proof}
    The full subcategory $\cE$ is extension-closed by \cref{lem:stab-positive-closure-properties,lem:j-left-special}.
    \cite[Proposition~4.25]{klemenc:stablehull} shows that the two exact structures on $\cE$ coincide, but one can also argue directly as follows: if $j(x) \xrightarrow{j(i)} j(y) \xrightarrow{j(p)} j(z)$ is a cofibre sequence in $\stab(\cE)$, it is also a cofibre sequence in $\stab_{\geq 0}(\cE)$.
    Since $\cE$ is left special in $\stab_{\geq 0}(\cE)$, the morphism $p$ becomes a projection after precomposition with some morphism $g \colon y' \to y$ in $\cE$.
    Since $j$ is fully faithful, $i$ is a fibre of $p$ in $\cE$.
    Quillen's ``obscure axiom''~\ref{lem:obscure-axiom} implies that $p$ is a projection, and we are done.
    
    Suppose that $\cE$ is weakly idempotent complete and consider a cofibre sequence $X \to j(y) \xrightarrow{j(p)} j(z)$ in $\stab_{\geq 0}(\cE)$.
    By \cref{lem:j-left-special}, there exist an exact sequence $x' \rightarrowtail y' \overset{q}{\twoheadrightarrow} z$  and a morphism $g \colon y' \to y$ in $\cE$ such that $q \simeq p \circ g$.
    This implies that the induced commutative square
    \[\begin{tikzcd}
        j(x')\ar[r]\ar[d] & j(y')\ar[d, "g"] \\
        X\ar[r] & j(y)
    \end{tikzcd}\]
    is a pushout.
    By \cref{lem:pushouts}, there exists an exact sequence $j(x') \to X \oplus j(y') \to j(y)$.
    Since the essential image of $j$ is closed under extensions, the object $X \oplus j(y')$ is equivalent to an object in $\cE$.
    It follows that $X$ is equivalent to an object in $\cE$ by weak idempotent completeness.

    Conversely, let $i \colon x \to y$ be a morphism in $\cE$ which admits a retraction $r \colon y \to x$.
    Then $\cofib(j(i))$ exists in $\stab_{\geq 0}(\cE)$, and we have $j(y) \simeq j(x) \oplus \cofib(j(i))$.
    Hence $j(r)$ is a projection in $\stab_{\geq 0}(\cE)$, and there exists an exact sequence $\cofib(j(i)) \to j(y) \xrightarrow{j(r)} j(x)$.
    Since $\cE$ is closed under fibres of projections, it follows that $\cofib(j(i))$ is equivalent to an object in $\cE$.
\end{proof}


\section{Keller's criterion}

Equipped with our understanding
of the Gabriel--Quillen embedding, we are now able to prove Keller's criterion by elementary methods of homological algebra.
The following notions are directly adapted from
\cite[Chapitre~II]{grothendieck:tohoku}.
Note that what we call a ``$\delta$-functor'' is an ``exact $\delta$-functor'' in Grothendieck's language.

\begin{defn}
    Let $\cE$ be an exact category.
    \begin{enumerate}
        \item A \emph{$\delta$-functor} on $\cE$ is a sequence of additive functors
         \[ \{F^n \colon \cE^\op \to \Ab\}_{n \geq 0} \]
         together with a sequence of homomorphisms
         \[ \{ \delta^n \colon F^n(x) \to F^{n+1}(z) \}_{n \geq 0} \]
         for each exact sequence $x \rightarrowtail y \twoheadrightarrow z$ in $\cE$ such that the following holds:
         \begin{enumerate}
             \item for each exact sequence $x \rightarrowtail y \twoheadrightarrow z$, the induced sequence
             \[ F^0(z) \to F^0(y) \to F^0(x) \xrightarrow{\delta^0} F^1(z) \to F^1(y) \to F^1(x) \xrightarrow{\delta^1} \ldots \]
            is exact;
            \item for each morphism of exact sequences
            \[\begin{tikzcd}
                x\ar[r, rightarrowtail]\ar[d] & y\ar[r, twoheadrightarrow]\ar[d] & z\ar[d] \\
                x'\ar[r, rightarrowtail] & y'\ar[r, twoheadrightarrow] & z'
            \end{tikzcd}\]
            in $\cE$ and each $n \geq 0$, the induced square
            \[\begin{tikzcd}
                F^n(x')\ar[r, "\delta^n"]\ar[d] & F^{n+1}(z')\ar[d] \\
                F^n(x)\ar[r, "\delta^n"] & F^{n+1}(z)
            \end{tikzcd}\]
            commutes.
        \end{enumerate}
        \item A \emph{transformation $\tau \colon F \Rightarrow G$ of $\delta$-functors} is a sequence of natural transformations $\{ \tau^n \colon F^n \Rightarrow G^n \}_{n \geq 0}$ such that for each exact sequence $x \rightarrowtail y \twoheadrightarrow z$ the induced square
        \[\begin{tikzcd}
            F^n(x)\ar[r, "\delta^n"]\ar[d, "\tau^n"'] & F^{n+1}(z)\ar[d, "\tau^{n+1}"] \\
            G^n(x)\ar[r, "\delta^n"] & G^{n+1}(z)
        \end{tikzcd}\]
        is commutative.
        \item A $\delta$-functor $F$ is \emph{universal} if for every $\delta$-functor $G$ on $\cE$, every natural transformation $F^0 \Rightarrow G^0$ extends uniquely to a transformation of $\delta$-functors $F \Rightarrow G$.
    \end{enumerate}
\end{defn}

A universal $\delta$-functor is uniquely determined by its $0$-th component.
The following lemma, which allows us to recognise universal $\delta$-functors, goes back to Grothendieck \cite[Proposition~2.2.1]{grothendieck:tohoku}.
See also \cite[Proposition~4.2]{buchsbaum:satellites} for the characterisation of universality using weakly effaceable functors.
We include a proof for the reader's convenience.

\begin{lem}\label[lem]{lem:universal-delta}
    Let $\cE$ be an exact category and let $F$ be a $\delta$-functor on $\cE$.
    If $F^n$ is weakly effaceable for every $n \geq 1$, then $F$ is universal.
\end{lem}
\begin{proof}
    Let $G$ be an arbitrary $\delta$-functor on $\cE$ and suppose that $\tau \colon F \Rightarrow G$ is a transformation of $\delta$-functors.
    Consider $\alpha \in F^{n+1}(x)$.
    Since $F^{n+1}$ is weakly effaceable, there exists an exact sequence $x'' \overset{i}{\rightarrowtail} x' \overset{p}{\twoheadrightarrow} x$ such that $p^*\alpha = 0$.
    From the commutativity of the diagram
    \[\begin{tikzcd}
        F^n(x'')\ar[r, "\delta^n"]\ar[d, "\tau^n_{x''}"] & F^{n+1}(x)\ar[r, "p^*"]\ar[d, " \tau^{n+1}_x"] & F^{n+1}(x')\ar[d, "\tau^{n+1}_{x'}"] \\
        G^n(x'')\ar[r, "\delta^n"] & G^{n+1}(x)\ar[r, "p^*"] & G^{n+1}(x')
    \end{tikzcd}\]
    and the exactness of the first row, it follows that there exists an element $\beta \in F^n(x'')$ with $\delta^n(\beta) = \alpha$.
    Hence $\tau^{n+1}_x(\alpha) = \delta^n(\tau^n_{x''}(\beta))$.
    By induction, we conclude that $\tau$ is uniquely determined by $\tau^0$.

    Moreover, this argument shows that if $\tau^n$ is already defined, then $\tau^{n+1}_x(\alpha) := \delta^n(\tau^n_{x''}(\beta))$ is the only possible definition for $\tau^{n+1}_x$.
    We need to show that this is well-defined and additive; compatibility with the coboundary maps $\delta^n$ holds by definition.

    If $\beta$ and $\gamma$ are two preimages of $\alpha$ under $\delta^n$, then their difference lies in the image of $F^n(x') \to F^n(x'')$.
    Since $\tau^n$ is a natural transformation, exactness of the lower row in the above diagram implies $\delta^n(\tau^n_{x''}(\beta - \gamma)) = 0$.

    Suppose that $y'' \overset{j}{\rightarrowtail} y' \overset{q}{\twoheadrightarrow} x$ is another exact sequence with $q^*\alpha = 0$.
    Since $F^{n+1}$ is additive, we then also have $(p + q)^*\alpha = 0 \in F^{n+1}(x' \oplus y')$.
    Let $\beta$ denote a preimage of $\alpha$ in $F^n(\fib(p+q))$, and define $\beta_x$ and $\beta_y$ to be the image of $\beta$ in $F^n(x'')$ and $F^n(y'')$, respectively.
    Then it follows from naturality of $\tau^n$ that
    \[ \delta^n(\tau^n_{x''}(\beta_x)) = \delta^n(\tau^n_{\fib(p+q)}(\beta)) = \delta^n(\tau^n_{y''}(\beta_y)).\]
    This shows that $\tau^{n+1}_x$ is well-defined.
    Additivity of $\tau^{n+1}_x$ follows immediately since the preceding discussion implies that we may choose the sum of two given preimages as a preimage of a sum in $F^{n+1}(x)$.

    We are left with showing that $\tau^{n+1}$ is a natural transformation.
    Let $f \colon x \to y$ be a morphism and let $\alpha \in F^n(y)$.
    Choose a projection $q \colon y'\twoheadrightarrow y$ such that $q^*\alpha = 0$,
    and pick a preimage $\beta$ of $\alpha$ under $\delta^n$.
    Define $p \colon x':= x \times_y y' \to x$ to be the pullback of $p$ along $f$, and let $f' \colon x' \to y'$ be the second structure map so that we obtain a morphism of exact sequences
    \[\begin{tikzcd}
        x''\ar[r, rightarrowtail]\ar[d, "f''"] & x'\ar[r, twoheadrightarrow, "p"]\ar[d, "f'"] & x\ar[d, "f"] \\
        y''\ar[r, rightarrowtail] & y'\ar[r, twoheadrightarrow, "q"] & y
    \end{tikzcd}\]
    Then the naturality of $\delta^n$ implies that $\delta^n((f'')^*\beta) = f^*\delta^n(\beta) = f^*\alpha$.
    From the naturality of both $\tau^n$ and $\delta^n$, we conclude that
    \[ \tau^{n+1}_x(f^*\alpha) = \delta^n(\tau^n_{x''}((f'')^*\beta)) = \delta^n((f'')^*\tau^n_{y''}(\beta)) = f^*\delta^n(\tau^n_{y''}(\beta)) = f^*\tau^{n+1}_y(\alpha).\qedhere \]
\end{proof}

\begin{lem}\label[cor]{cor:rhom-effacebale}
    Let $\cE$ be an exact category.
    Then
    \[ \RHom^n_\cE(-,y) := \pi_0\Hom_{\stab(\cE)}(-,\Sigma^n j(y)) \colon \cE^\op \to \Ab \]
    is weakly effaceable for every $n \geq 1$.
\end{lem}
\begin{proof}
    The calculus of fractions formula for Verdier quotients implies that an arbitrary morphism $j(x) \to \Sigma^n j(y)$ can be represented by a zig-zag
    \[ \yo(x) \xrightarrow{f} Y \xleftarrow{s} \Sigma^n\yo(y) \]
    such that $\cofib(s)$ is acyclic.
    Since $\pi_0\cofib(s)$ is weakly effaceable by \cref{prop:acyclics}, there exists a projection $p \colon x' \twoheadrightarrow x$ in $\cE$ such that $f \circ p$ factors over $\Sigma^n\yo(y)$.
    Suspensions of additive presheaves are computed pointwise, so $\Sigma^n\yo(y)$ is $n$-connective. This implies that every morphism $\yo(x') \to \Sigma^n\yo(y)$ is trivial.
\end{proof}

\begin{cor}\label[cor]{cor:rhom-universal}
 Let $\cE$ be an exact category. Then the sequence
 \[ \RHom_\cE(-,y) := \left\{ \RHom^n_\cE(-,y) \right\}_{n \geq 0} \]
 refines to a universal $\delta$-functor.
\end{cor}
\begin{proof}
 Let $x_0 \overset{i}{\rightarrowtail} x_1 \overset{p}{\twoheadrightarrow} x_2$ be an exact sequence in $\cE$.
 Then $j(x_0) \to j(x_1) \to j(x_2)$ is a fibre sequence in $\stab(\cE)$, and we obtain 
 \[ \delta^n \colon \RHom^n_\cE(x_2,y) \to \RHom^{n+1}_\cE(x_0,y) \]
 as the boundary map on the homotopy groups of the induced fibre sequence of mapping spectra
 \[ \hom_{\stab(\cE)}(j(x_0),j(y)) \to \hom_{\stab(\cE)}(j(x_1),j(y)) \to \hom_{\stab(\cE)}(j(x_2),j(z)). \]
 This equips $\RHom_\cE(-,y)$ with the structure of a $\delta$-functor.
\cref{cor:rhom-effacebale} shows that $\RHom^n_\cE(-,y)$ is weakly effaceable for all $n \geq 1$, so universality follows from \cref{lem:universal-delta}.
\end{proof}

The following generalises \cite[Theorem~12.1]{keller:derived-cats}.
See also \cite[Expos\'e VIII, Proposition~4.1]{SGA5} for the case of abelian categories.

\begin{theorem}[Keller's criterion]\label[theorem]{thm:keller}
 Let $\cE$ be an exact category. 
 If $\cU \subseteq \cE$ is a left special subcategory of $\cE$, then the induced functor $\stab(\cU) \to \stab(\cE)$ is fully faithful.
\end{theorem}
\begin{proof}
    Denote by $\widehat\cU$ the full subcategory of $\stab(\cU)$ given by those objects $U$ such that the induced map
    \[ \hom_{\stab(\cU)}(j_\cU(-),U) \to \hom_{\stab(\cE)}(j_\cE(-),\stab(\inc)(U)) \]
    of mapping spectra is an equivalence.
    Since we require an equivalence of mapping spectra, $\widehat\cU$ is a full stable subcategory of $\stab(\cU)$.
    In order to show $\widehat\cU = \stab(\cU)$, it is therefore enough to show that $\widehat\cU$ contains the essential image of $j_\cU$.

    Let $v \in \cU$, so that $\stab(\inc)(j_\cU(v)) \simeq j_\cE(v)$. Then the induced map
    \[ \Hom_{\stab(\cU)}(j_\cU(-),j_\cU(v)) \to \Hom_{\stab(\cE)}(j_\cE(-),j_\cE(v)) \]
    of mapping anima is an equivalence because $j$ is assumed to be fully faithful.
    We are left with showing that this map also induces an isomorphism on negative homotopy groups.
    Since $\cU$ is a left special subcategory of $\cE$, the restricted functors
    \[ \RHom^n_\cE(\inc(-),v) \colon \cU^\op \to \Ab \]
    are weakly effaceable for $n \geq 1$ by \cref{cor:rhom-effacebale}, so $\inc^*\RHom$ is a universal $\delta$-functor on $\cU$ by \cref{lem:universal-delta}.
    As $\RHom^0_\cU(-,v) \simeq \inc^*\RHom^0_\cE(-,v)$ by assumption, it follows that $\widehat\cU$ contains the essential image of $j_\cU$.
\end{proof}

\begin{cor}[{\cite[Theorem~2.9]{saunier:heart}}]\label[cor]{cor:counit-equiv}
    Let $\cC$ be a stable category with a bounded heart structure.
    Then the induced functor $\stab(\cC^\heartsuit) \to \cC$ is an equivalence.
\end{cor}
\begin{proof}
    Recall from \cite[Definition~1.3]{saunier:heart} that the inclusion $\cA \to \cB$ of an extension-closed subcategory into an exact category is \emph{resolving} if the following holds:
    \begin{enumerate}
        \item for every exact sequence $x \rightarrowtail y \twoheadrightarrow z$ with $y \in \cA$ we also have $x \in \cA$;
        \item for every object $z \in \cB$ there exists a projection $y \twoheadrightarrow z$ with $y \in \cA$. 
    \end{enumerate}
    In particular, such a subcategory is left special: given an exact sequence $x \rightarrowtail y \twoheadrightarrow z$ with $z \in \cA$, pick a projection $y' \twoheadrightarrow y$ with $y \in \cA$; then the first condition implies that the composite $y' \twoheadrightarrow z$ is a projection in $\cA$.

    The proof of \cite[Theorem~2.9]{saunier:heart} shows that $\cC^\heartsuit \to \cC_{\geq 0}$ is a composition of resolving functors.
    Since the heart structure is bounded, it follows by induction that $\cC^\heartsuit$ is left special in $\cC_{\geq 0}$.
    \cref{thm:keller} implies that $\stab(\cC^\heartsuit) \to \cC$ is fully faithful.
    Since $\cC$ is generated by $\cC^\heartsuit$, it follows that this functor is also essentially surjective.
\end{proof}

There is an evident dual version of Keller's criterion.

\begin{defn}
    An extension-closed subcategory $\cU$ of an exact category $\cE$ is \emph{right special} if $\cU^\op \subseteq \cE^\op$ is left special.
\end{defn}

\begin{cor}\label[cor]{cor:keller-right-special}
    If $\cU \subseteq \cE$ is a right special subcategory of the exact category $\cE$, then $\stab(\cU) \to \stab(\cE)$ is fully faithful.
\end{cor}
\begin{proof}
    This follows from \cref{thm:keller} and \cref{cor:stab-op}.
\end{proof}

\begin{rem}
    It is not true without further assumptions that the inclusion of an extension-closed subcategory induces a fully faithful functor on bounded derived categories.
    Examples are provided by stable $\infty$-categories which admit a t-structure, but do not contain the bounded derived category of their heart as a full subcategory.
    Concretely, one can take for $\cC$ the stable $\infty$-category of bounded spectra with finitely generated homotopy groups and for $\cU$ the full subcategory of Eilenberg--MacLane spectra of finitely generated abelian groups.

    Moreover, there exist extension-closed subcategories which are neither left nor right special, but still induce fully faithful functors on stabilisations.
   For example, consider an additive category $\cA$ as an extension-closed subcategory of its stabilisation $\stab(\cA)$.
   For any $a \in \cA$, the map $0 \to a$ is a projection in $\stab(\cA)$ because $\stab(\cA)$ is stable.
   However, no projection $b \twoheadrightarrow a$ in $\cA$ factors over $0$ unless $a \simeq 0$---this also shows that it is important to consider only the positive part of the heart structure on $\stab(\cA)$ in \cref{lem:j-left-special}.
   Similarly, no inclusion in $\cA$ can factor over the zero map $0 \colon a \to 0$ for non-trivial $a$.
   However, the induced functor $\stab(\cA) \to \stab(\stab(\cA)) \simeq \stab(\cA)$ is an equivalence.
\end{rem}

One can also use \cref{thm:keller} together with the Gabriel--Quillen embedding to show that every exact category admits a weak idempotent completion.

\begin{lem}\label[lem]{lem:stab-weakly-idempotent-complete}
    Let $\cE$ be an exact category. Define
    \[ \cE^\flat := \{ X \in \stab(\cE) \mid \exists\,x,x'\in\cE \colon X \oplus j(x') \simeq j(x) \}. \]
    Then $\cE^\flat$ is an extension-closed subcategory of $\stab(\cE)$, and the restriction functor
    \[ \Fun^\exct(\cE^\flat,\cD) \to \Fun^\exct(\cE,\cD) \]
    is an equivalence for every weakly idempotent complete exact category $\cD$.
    
    In particular, the induced exact functor $\stab(\cE) \to \stab(\cE^\flat)$ is an equivalence.
\end{lem}
\begin{proof}
    Let $X \xrightarrow{i} Y \xrightarrow{p} Z$ be a cofibre sequence in $\stab(\cE)$ with $X,Z \in \cE^\flat$.
    Choose $x,x',z,z' \in \cE$ such that $X \oplus j(x') \simeq j(x)$ and $Z \oplus j(z') \simeq j(z)$.
    Then
    \[ X \oplus x' \oplus 0 \xrightarrow{i \oplus \id \oplus 0} Y \oplus x' \oplus z' \xrightarrow{p \oplus 0 \oplus \id} Z \oplus 0 \oplus z' \]
    is also a cofibre sequence.
    Since $\cE$ is closed under extensions in $\stab(\cE)$ by \cref{lem:stab-positive-closure-properties} and \cref{thm:gabriel-quillen}, it follows that $Y \oplus x' \oplus z'$ also lies in the essential image of $j$.
    Hence $\cE^\flat$ is extension-closed in $\stab(\cE)$.

    Now consider the commutative diagram
    \[\begin{tikzcd}[row sep=1.5em]
        \Fun^\exct(\cE^\flat,\cD)\ar[r]\ar[d] & \Fun^\exct(\cE,\cD)\ar[d] \\
        \Fun^\exct(\cE^\flat,\stab(\cD))\ar[r] & \Fun^\exct(\cE,\stab(\cD)) \\
        \Fun^\exct(\stab(\cE^\flat),\stab(\cD))\ar[r]\ar[u, "\simeq"] & \Fun^\exct(\stab(\cE),\stab(\cD))\ar[u, "\simeq"']
    \end{tikzcd}\]
    The lower vertical morphisms are equivalences by the universal property of $\stab$.
    Since $\cE \subseteq \cE^\flat \subseteq \stab(\cE)$ are inclusions of extension-closed subcategories and $\cE$ is left special in $\stab(\cE)$, it follows that $\cE$ is also left special in $\cE^\flat$.
    \cref{thm:keller} implies that $\stab(\cE) \to \stab(\cE^\flat)$ is an equivalence,
    and it follows that the bottom and middle horizontal arrows are equivalences.

    The upper vertical morphisms are fully faithful.
    Since $\cD$ is weakly idempotent complete, an exact functor $\cE^\flat \to \stab(\cD)$ restricts to an exact functor $\cE \to \cD$ if and only if it restricts to an exact functor $\cE^\flat \to \cD$.
    This means that the upper square is a pullback, so the lemma follows.
\end{proof}


\section{Bounded heart categories}\label[section]{sec:bounded-heart-categories}

This section contains the proofs of \cref{thm:main,thm:gillet-waldhausen}.
Our main goal is to show that $\stab(\cE)$ always carries a bounded heart structure.

\begin{defn}
 Let $\cE$ be an exact category and let $X \in \stab(\cE)$.
 \begin{enumerate}
  \item The object $X$ \emph{admits a length $0$ resolution} if it is equivalent to an object in the essential image of $j$;
  \item the object $X$ \emph{admits a length $n+1$ resolution} if there exists a cofibre sequence $X' \to j(x) \to X$ in $\stab(\cE)$ such that $X'$ admits a length $n$ resolution. 
 \end{enumerate}
\end{defn}

\begin{lem}\label[lem]{lem:stab-resolutions}
 For each $X \in \stab_{\geq 0}(\cE)$ there exists some $n \geq 0$ such that $X$ admits a length $n$ resolution.
\end{lem}
\begin{proof}
	The required resolutions exist already in $\stabadd(\cE)$ because this category carries a bounded weight structure by \cref{prop:weight-structure}.
\end{proof}

\begin{lem}\label[lem]{lem:resolutions-cofibre-sequences}
    Let $\cE$ be a weakly idempotent complete exact category and let $n \geq 0$.
    Consider a cofibre sequence
    \[ X \to Y \to Z \]
    in $\stab_{\geq 0}(\cE)$. Then the following holds:
    \begin{enumerate}
        \item\label{lem:resolutions-cofibre-sequences-1} If $Y$ admits a length $n$ resolution and $Z$ admits a length $n+1$ resolution, then $X$ admits a length $n$ resolution.
        \item\label{lem:resolutions-cofibre-sequences-2} If $X$ and $Z$ admit length $n+1$ resolutions, then $Y$ admits a length $n+1$ resolution.
        \item\label{lem:resolutions-cofibre-sequences-4} If $X$ and $Y$ admit length $n$ resolutions, then $Z$ admits a length $n+1$ resolution.
        \item\label{lem:resolutions-cofibre-sequences-3} If $Y$ and $Z$ admit length $n+1$ resolutions, then $X$ admits a length $n+1$ resolution.
    \end{enumerate}
\end{lem}
\begin{proof}
    The proof assertions~\eqref{lem:resolutions-cofibre-sequences-1}, \eqref{lem:resolutions-cofibre-sequences-2} and \eqref{lem:resolutions-cofibre-sequences-3} can be copied almost verbatim from \cite[Lemma on p 102]{quillen:higher-k}.

    We prove all statements simultaneously by induction on $n$.
    For \eqref{lem:resolutions-cofibre-sequences-1}, choose a cofibre sequence $Z' \to j(z) \to Z$ such that $Z'$ admits a resolution of length $n$ and let $\overline{Y} := Y \times_Z j(z)$.
    Then $\overline{Y}$ is an extension of objects admitting resolutions of length $n$.
    If $n=0$, the object $\overline{Y}$ also admits a resolution of length $0$ by \cref{thm:gabriel-quillen}; for $n > 0$ it follows by induction from \eqref{lem:resolutions-cofibre-sequences-2} that $\overline{Y}$ admits a length $n$ resolution because there exists a cofibre sequence $Z' \to \overline{Y} \to Y$.
    From the cofibre sequence $X \to \overline{Y} \to j(z)$, we conclude that $X$ also admits a length $n$ resolution: if $n=0$, this holds by \cref{thm:gabriel-quillen} because $\cE$ is weakly idempotent complete, otherwise this follows from the inductive assumption.

    For \eqref{lem:resolutions-cofibre-sequences-2}, choose a cofibre sequence $Z'\to j(z) \to Z$ as above and define $\overline{Y}$ as before.
    Since $\cE$ is left special in $\stab_{\geq 0}(\cE)$ by \cref{lem:j-left-special}, there exists a projection $p \colon \overline{y} \twoheadrightarrow z$ such that $j(p)$ factors over $\overline{Y} \to j(z)$.
    The fibre of the composite projection $j(\overline{y}) \to Z$ is an extension of $Z'$ and a representable object, so it admits a length $n$ resolution by the inductive assumption.
    Now choose a cofibre sequence $X'\to j(x) \to X$ such that $X'$ admits a length $n$ resolution.
    Then the transformation of cofibre sequences
    \[\begin{tikzcd}
        j(x)\ar[r]\ar[d] & j(x \oplus \overline{y})\ar[r]\ar[d] & j(\overline{y})\ar[d] \\
        X\ar[r] & Y\ar[r] & Z
    \end{tikzcd}\]
    is a pointwise projection by \cref{lem:five-lemma}.
    The fibre of $j(x \oplus \overline{y}) \to Y$ is an extension of objects admitting length $n$ resolutions.
    The inductive assumption implies that $Y$ admits a length $n+1$ resolution.

    For \eqref{lem:resolutions-cofibre-sequences-4}, the claim is obviously true for $n=0$.
    Otherwise, choose a cofibre sequence $Y'\to j(y) \to Y$ such that $Y'$ admits a length $n$ resolution.
    Then the fibre of the composite projection $j(y) \to Z$ is an extension of $X$ by $Y'$, so \eqref{lem:resolutions-cofibre-sequences-2} implies that $Z$ admits a length $n+1$ resolution.

    For \eqref{lem:resolutions-cofibre-sequences-3}, choose again a cofibre sequence $Y'\to j(y) \to Y$ such that $Y'$ admits a length $n$ resolution.
    Then \eqref{lem:resolutions-cofibre-sequences-1} implies that the fibre of the composite projection $j(y) \to Z$ admits a resolution of length $n$.
    It follows from \eqref{lem:resolutions-cofibre-sequences-4} and the cofibre sequence $Y' \to \fib(j(y) \to Z)) \to X$ that $X$ admits a resolution of length $n+1$.
\end{proof}

Recall once more from \cref{prop:weight-structure} that $\stabadd(\cE)$ carries a bounded weight structure with $\stabadd_{\geq 0}(\cE) = \pshadd^\f(\cE)$, and that we denote the essential image of $\stabadd_{\geq 0}(\cE)$ in $\stab(\cE)$ by $\stab_{\geq 0}(\cE)$.

\begin{prop}\label[prop]{prop:stab-bounded-heart-structure}
 Let $\cE$ be an exact category.
 Then $(\stab_{\geq 0}(\cE),\stab_{\leq 0}(\cE))$ is a bounded heart structure on $\stab(\cE)$, where $\stab_{\leq 0}(\cE)$ denotes the essential image of $\stabadd_{\leq 0}(\cE)$ under the Verdier projection $\stabadd(\cE) \to \stab(\cE)$.
 
 The induced exact functor $\cE \to \stab(\cE)^\heartsuit$ exhibits the heart as a weak idempotent completion of $\cE$.
\end{prop}
\begin{proof}
 The full subcategory $\stab_{\geq 0}(\cE)$ is closed under extensions and finite colimits by \cref{lem:stab-positive-closure-properties}, and every object $X \in \stab(\cE)$ fits into a cofibre sequence $X_{\leq 0} \to X \to \Sigma X_{\geq 0}$ with $X_{\geq 0} \in \stab_{\geq 0}(\cE)$ and $X_{\leq 0} \in \stab_{\leq 0}(\cE)$ since such a decompostion already exists in $\stabadd(\cE)$.
 We have to show that $\stab_{\leq 0}(\cE)$ is closed under extensions and finite limits.
 
 Consider an arbitrary exact category $\cD$.
 By \cref{cor:stab-op}, the canonical functor $\stab(\cD^\op) \to \stab(\cD)^\op$ is an equivalence.
 This applies in particular to the split-exact structure on $\cD$, so we also have an equivalence $\stabadd(\cD^\op) \xrightarrow{\sim} \stabadd(\cD)^\op$.
 By Sosnilo's theorem \cite[Corollary~3.4]{sosnilo:heart}, this is in fact an equivalence of weighted categories, and therefore restricts to an equivalence $\stabadd_{\leq 0}(\cD^\op) \simeq \stabadd_{\geq 0}(\cD)^\op$.
 From the commutativity of the diagram
 \[\begin{tikzcd}
	\stabadd(\cD^\op)\ar[r, "\simeq"]\ar[d] & \stabadd(\cD)^\op\ar[d] \\
	\stab(\cD^\op)\ar[r, "\simeq"] & \stab(\cD)^\op
 \end{tikzcd}\]
 we conclude that the lower horizontal arrow restricts to an equivalence
 \[ \stab_{\leq 0}(\cD^\op) \simeq \stab_{\geq 0}(\cD)^\op. \]
 Choosing $\cD = \cE^\op$, we conclude that $\stab_{\leq 0}(\cE)$ is closed under extensions and finite limits in $\stab(\cE)$.

 Since $(\stabadd_{\geq 0}(\cE),\stabadd_{\leq 0}(\cE))$ is a bounded weight structure by \cref{prop:weight-structure} and $\stabadd(\cE) \to \stab(\cE)$ is essentially surjective,
 the heart structure $(\stab_{\geq 0}(\cE),\stab_{\leq 0}(\cE))$ is bounded.
 
 By definition, the functor $j \colon \cE \to \stab(\cE)$ factors over the heart of this heart structure.
 Since $\stab(\cE)^\heartsuit$ is weakly idempotent complete, we obtain an induced fully faithful functor $\cE^\flat \to \stab(\cE)^\heartsuit$.
 Since $\stab_{\geq 0}(\cE)$ is generated by $\cE$ under finite colimits and $\stab_{\leq 0}(\cE)$ is generated by $\cE$ under finite limits (and similarly for $\cE^\flat)$, it follows that the inclusion $\cE \subseteq \cE^\flat$ induces an equivalence of bounded heart categories $(\stab(\cE),\stab_{\geq 0}(\cE),\stab_{\leq 0}(\cE)) \xrightarrow{\sim} (\stab(\cE^\flat),\stab_{\geq 0}(\cE^\flat),\stab_{\leq 0}(\cE^\flat))$.
 In particular, it suffices to show that $\cE \to \stab(\cE)^\heartsuit$
 is an equivalence whenever $\cE$ is weakly idempotent complete.

 So assume that $\cE$ is weakly idempotent complete and let $X$ be an object in $\stab(\cE)^\heartsuit$.
 Then \cref{lem:stab-resolutions} together with the preceding discussion implies that $X$ admits a length $n$ resolution in $\stab_{\leq 0}(\cE)^\op$ for some $n \geq 1$.
 If $n = 0$, then $X$ lies in the essential image of $j$ by definition.
 If $n \geq 1$, we find an exact sequence $X \to j(x) \to X'$ in $\stab_{\leq 0}(\cE)$ such that $X'$ admits a length $n-1$ resolution.
 By induction, $X'$ lies in the essential image of $j$.
 Since $X$ also lies in $\stab_{\geq 0}(\cE)$ and $\cE$ is closed under fibres of projections in $\stab_{\geq 0}(\cE)$ by \cref{thm:gabriel-quillen}, it follows that $X$ lies in the essential image of $j$.
\end{proof}

\begin{cor}\label[cor]{cor:stab-heart-universal-prop}
    Let $\cE$ be an exact category and let $(\cC,\cC_{\geq 0},\cC_{\leq 0})$ be a stable category with a heart structure.
    Then the restriction functor
    \[ j^* \colon \Fun^\exct((\stab(\cE),\stab_{\geq 0}(\cE),\stab_{\leq 0}(\cE),(\cC,\cC_{\geq 0},\cC_{\leq 0})) \to \Fun^\exct(\cE,\cC^\heartsuit) \]
is an equivalence.
\end{cor}
\begin{proof}
   The triple $(\stab(\cE),\stab_{\geq 0}(\cE),\stab_{\leq 0}(\cE))$ is a bounded heart structure by \cref{prop:stab-bounded-heart-structure}.
    Consider the commutative square
   \[\begin{tikzcd}
      \Fun^\exct((\stab(\cE),\stab_{\geq 0}(\cE),\stab_{\leq 0}(\cE)),(\cC,\cC_{\geq 0},\cC_{\leq 0}))\ar[d]\ar[r, "j^*"] & \Fun^\exct(\cE,\cC^\heartsuit)\ar[d] \\
      \Fun^\exct(\stab(\cE),\cC)\ar[r, "j^*"] & \Fun^\exct(\cE,\cC)
   \end{tikzcd}\]
   The universal property of $\stab$ implies that the lower horizontal arrow is an equivalence, and both vertical arrows are fully faithful.

  Suppose that an exact functor $f \colon \stab(\cE) \to \cC$ restricts to an exact functor $\cE \to \cC^\heartsuit$.
   Since $\stab_{\geq 0}(\cE)$ is generated by $\cE$ under finite colimits, it follows that $f(\stab_{\geq 0}(\cE)) \subseteq \cC_{\geq 0}$.
   Dually, $f(\stab_{\leq 0}(\cE)) \subseteq \cC_{\leq 0}$ because $\stab_{\leq 0}(\cE)$ is generated by $\cE$ under finite limits.

   Conversely, a heart-exact functor $f \colon \stab(\cE) \to \cC$ evidently restricts to a functor $\stab(\cE)^\heartsuit \to \cC$, and therefore also to an exact functor $\cE \to \cC^\heartsuit$.

   This means precisely that the above square is a pullback, which proves the corollary. 
\end{proof}

The main theorems follow by combining the preceding results.

\begin{proof}[Proof of \cref{thm:main}]
    \cref{cor:stab-heart-universal-prop} shows that the functor $(-)^\heartsuit \colon \catheart \to \exact$ admits a left adjoint.
    Let $(\cC,\cC_{\geq 0},\cC_{\leq 0})$ be a bounded heart category.
    The induced functor $\stab(\cC^\heartsuit) \to \stab(\cC)$ is an equivalence by \cref{cor:counit-equiv} and restricts to an equivalence on hearts by \cref{prop:stab-bounded-heart-structure} because hearts of heart structures are weakly idempotent complete.
    Since the positive and negative part of a bounded heart structure are generated by the heart under finite colimits and limits, respectively, it follows that the counit of this adjunction is an equivalence. The characterisation of the essential image of $(-)^\heartsuit$ also follows from \cref{prop:stab-bounded-heart-structure}.

    To see that this adjunction restricts to Sosnilo's adjunction between additive categories and categories with a bounded weight structure, observe that the category $\Ac(\cA)$ vanishes for an additive category $\cA$ with the split-exact structure.
    Hence the heart structure from \cref{prop:stab-bounded-heart-structure} is the canonical weight structure on $\stabadd(\cA)$.
\end{proof}    

\begin{proof}[Proof of \cref{thm:gillet-waldhausen}]
 The inclusion $\cE \to \cE^\flat$ induces an equivalence in K-theory by the cofinality theorem \cite[Theorem~10.19]{barwick:ktheory}. 
 Applying \cref{lem:stab-weakly-idempotent-complete}, we may assume without loss of generality that $\cE$ is weakly idempotent complete.
 
 Since $\stab_{\geq 0}(\cE) \to \stab(\cE)$ is a Spanier--Whitehead stabilisation:
 the induced functor $\SW(\stab_{\geq 0}(\cE)) \to \stab(\cE)$ is fully faithful because $\stab_{\geq 0}(\cE)$ is prestable, and it is essentially surjective because the heart structure on $\stab(\cE)$ is bounded below.
 Therefore, it suffices to show that the inclusion functor $\cE \to \stab_{\geq 0}(\cE)$ induces an equivalence in K-theory.
 Note that $\cE$ is closed under extensions and fibres of projections in $\stab_{\geq 0}(\cE)$ by \cref{thm:gabriel-quillen}.
 
 For $n \geq 0$, denote by $\cE_n$ the full subcategory of objects in $\stab(\cE)$ which admit a length $n$ resolution.
 By \cref{lem:resolutions-cofibre-sequences}, the full subcategory
 $\cE_n$ is closed under extensions and fibres of projections in $\cE_{n+1}$.
The resolution theorem \cite[Theorem~1.4]{saunier:heart} therefore implies that $\K(\cE_n) \to \K(\cE_{n+1})$ is an equivalence for all $n$.
Since $\bigcup_n \cE_n = \stab_{\geq 0}(\cE)$ by \cref{lem:stab-resolutions} and K-theory commutes with filtered colimits, it follows that $\K(\cE) \to \K(\stab_{\geq 0}(\cE))$ is an equivalence.
\end{proof}


\section{The universal property of K-theory of exact categories}

The goal of this section is to prove a universal property of the connective K-theory of exact categories by leveraging the previous parts as well as work of \cite{bgt:univcharkt}. We begin by introducing our version of additivity in the context of exact categories.

Recall that $S_2(\cE)$, the degree $2$ part of the Waldhausen $S_\bullet$-construction, is an exact category whenever $\cE$ is an exact category.
Thinking of $S_2(\cE)$ as the category of exact sequences $x \rightarrowtail y \twoheadrightarrow z$ in $\cE$, inclusions and projections are defined pointwise (see \cref{lem:five-lemma}).
Given two extension-closed subcategories $\cA$ and $\cB$ of $\cE$, denote by $E(\cA,\cE,\cB) \subseteq S_2(\cE)$ the full subcategory on the exact sequences $a \rightarrowtail x \twoheadrightarrow b$ satisfying $a \in \cA$ and $b \in \cB$.

\begin{defn}\label[defn]{defn:sod}
    Let $\cE$ be an exact category.
    An ordered pair $(\cA,\cB)$ of extension-closed subcategories of $\cE$ is a \textit{semi-orthogonal decomposition} of $\cE$ if the functor
        \[ e \colon E(\cA,\cE,\cB) \to \cE,\ (a \rightarrowtail x \twoheadrightarrow b) \mapsto x \]
        is an equivalence of exact categories.
\end{defn}

Let us record a number of equivalent characterisations of semi-orthogonal decompositions.

\begin{prop}\label[lem]{lem:sod}
    Let $\cE$ be an exact category, and let $(\cA,\cB)$ be an ordered pair of extension-closed subcategories of $\cE$.
    Then the following are equivalent:
    \begin{enumerate}
        \item\label{lem:sod-i} the pair $(\cA,\cB)$ is a semi-orthogonal decomposition;
        \item\label{lem:sod-ii} it holds that:
      \begin{enumerate}
        \item\label{lem:sod-iia} the inclusion functor $i \colon \cA \to \cE$ admits a right adjoint $p$ and the inclusion functor $j \colon \cB \to \cE$ admits a left adjoint $q$;
        \item\label{lem:sod-iib} the functors $p$ and $q$ are exact;
        \item\label{lem:sod-iic} the sequences $\cA \xrightarrow{i} \cE \xrightarrow{q} \cB$ and $\cB \xrightarrow{j} \cE \xrightarrow{p} \cA$ are bifibre sequences in $\exact$;
        \item\label{lem:sod-iid} the counit $ip(x) \to x$ is an inclusion and the unit $x \to jq(x)$ is a projection for all $x \in \cE$.
      \end{enumerate}
        \item\label{lem:sod-iii} the inclusion functor $i \colon \cA \to \cE$ admits an exact right adjoint $p$ such that the counit $ip(x) \to x$ is an inclusion for all $x \in \cE$, and $\cB = \ker(p)$;
        \item\label{lem:sod-iv} the inclusion functor $j \colon \cB \to \cE$ admits an exact left adjoint $q$ such that the unit $x \to jq(x)$ is a projection for all $x \in \cE$, and $\cA = \ker(q)$;
        \item\label{lem:sod-v} it holds that:
         \begin{enumerate}
             \item\label{lem:sod-va} for every $x\in\cE$, there exists an exact sequence $a\rightarrowtail x\twoheadrightarrow b$ with $a\in\cA$ and $b\in\cB$;
             \item\label{lem:sod-vb} for every $a\in\cA$ and $b\in\cB$, we have $\Hom_\cE(a, b)\simeq 0$;
             \end{enumerate}
             In particular, the exact sequence from \eqref{lem:sod-va} is functorial in $x$. It will be denoted by $p(x) \rightarrowtail x \twoheadrightarrow q(x)$ in the sequel.
             \begin{enumerate}
             \setcounter{enumii}{2}
             \item\label{lem:sod-vc} for every projection $y \twoheadrightarrow z$ in $\cE$, the induced map $p(y) \to p(z)$ is a projection in $\cA$;
         \end{enumerate}
        \item\label{lem:sod-vi} there exist exact functors $p \colon \cE \to \cA$ and $q \colon \cE \to \cB$ and an exact sequence of functors $p \rightarrowtail \id \twoheadrightarrow q$ in $\cE$ such that the inclusion map is the counit of an adjunction and the projection is the unit of an adjunction;
    \end{enumerate}
\end{prop}
\begin{proof}
    Suppose $(\cA,\cB)$ is a semi-orthogonal decomposition.
    To prove $\eqref{lem:sod-i} \Rightarrow \eqref{lem:sod-ii}$, it suffices to check that $E(\cA,\cE,\cB)$ always satisfies the properties of $\eqref{lem:sod-ii}$.

    The exact functor $i' \colon \cA \to \cE(\cA,\cE,\cB)$ given on objects
    by $a \mapsto (a \overset{\id}{\rightarrowtail} a \twoheadrightarrow 0)$, has an exact right adjoint $p'$ given by $(a \rightarrowtail x \twoheadrightarrow b) \mapsto a$, and the counit of this adjunction is given by the canonical transformation
    \[\begin{tikzcd}
        a\ar[r, tail, "\id"]\ar[d, tail] & a\ar[r, two heads]\ar[d, tail] & 0\ar[d, tail] \\
        a\ar[r, tail] & x\ar[r, two heads] & b
    \end{tikzcd}\]
    Moreover, $e \circ i' \simeq i$ and thus $p := p' \circ e^{-1}$ is the desired right adjoint to $i$.
    By considering opposite categories, we also have that $j' \simeq e^{-1} \circ j \colon \cB \to E(\cA,\cE,\cB)$ admits an exact left adjoint $q'$ whose unit is a projection.
    Hence, to prove $\eqref{lem:sod-i} \Rightarrow \eqref{lem:sod-ii}$, it suffices to check that $\cE(\cA,\cE,\cB)$ always satisfies the properties of $\eqref{lem:sod-ii}$.
    
    The above has already showed $\eqref{lem:sod-iia}$, $\eqref{lem:sod-iib}$ and $\eqref{lem:sod-iid}$. It also follows that $\cB$ is the localisation of $\cE$ at the unit morphisms and $\cA$ is the localisation of $\cE$ at the counit morphisms.
    Observing that $q'$ vanishes precisely on the objects in the essential image of $i'$, it follows that $q'$ is a localisation at the collection of projections with fibre in the essential image of $i'$, and dually for $p'$ and $j'$. In particular, the sequences of $\eqref{lem:sod-iic}$ are fibre sequences in $\exact$ since those are computed underlying.

    If $f \colon E(\cA,\cE,\cB) \to \cD$ is an exact functor vanishing on the essential image of $i'$, it necessarily inverts all projections whose fibre lies in $\cA$, so it induces an exact functor $\cB \to \cD$.
    This shows that restriction along $q'$ is a fully faithful functor $\Fun^\exct(\cB,\cD) \to \Fun^ \exct(E(\cA,\cE,\cB),\cD)$ with essential image those exact functors vanishing on the essential image of $i'$.
    In particular, $q'$ is the cofibre of $i'$ in $\exact$ so that $\cA \xrightarrow{i'} E(\cA,\cE,\cB) \xrightarrow{q'} \cB$ is a bifibre sequence.
    Passing to opposite categories, we also have that $\cB \xrightarrow{j'} E(\cA,\cE,\cB) \xrightarrow{p'} \cA$ is a bifibre sequence.
    This proves $\eqref{lem:sod-i} \Rightarrow \eqref{lem:sod-ii}$.

    Obviously, \eqref{lem:sod-ii} implies both \eqref{lem:sod-iii} and \eqref{lem:sod-iv}.
    
    Let us show $\eqref{lem:sod-iii} \Rightarrow \eqref{lem:sod-iv}$.
    Setting $q(x) := \cofib(ip(x) \rightarrowtail x)$ to be the cofibre of the unit morphism, it is straightforward to check that this yields a left adjoint to the inclusion $j \colon \cB \to \cE$.
    Since it preserves colimits, exactness of $q$ will follow from $q$ preserving inclusions.
    If $x \rightarrowtail y$ is an inclusion, consider the induced commutative diagram
    \[\begin{tikzcd}
            p(x)\ar[r, tail]\ar[d, tail] & p(y)\ar[r, two heads]\ar[d, tail] & p(z)\ar[d, tail] \\
            x\ar[r, tail]\ar[d, two heads] & y\ar[r, two heads]\ar[d, two heads] & z\ar[d, two heads] \\
            q(x)\ar[r] & q(y)\ar[r] & q(z)
    \end{tikzcd}\]
    whose upper row is exact by assumption.
    By \cref{lem:five-lemma}, it follows that $x \cup_{p(x)} p(y) \to y$ is an inclusion, and therefore $q(x) \to q(y)$ is also an inclusion.
    
    Moreover, an object $x \in \cE$ lies in $\cA$ if and only if $ip(x) \to x$ is an equivalence, which happens precisely if its cofibre is trivial.
    Since this cofibre computes $q(x)$, it follows that $\cA = \ker(q)$.

    By considering opposite categories, we also have $\eqref{lem:sod-iv} \Rightarrow \eqref{lem:sod-iii}$.

    Assuming \eqref{lem:sod-iii}, we have just seen that \eqref{lem:sod-va} holds,
    and \eqref{lem:sod-vc} follows from the exactness of $p$.
    Since $p$ is right adjoint to the inclusion of $\cA$, we have for $a \in \cA$ and $b \in \cB$ that
    \[ \Hom_\cE(a,b) \simeq \Hom_\cA(a,p(b)) \simeq 0 \]
    because $\cB = \ker(p)$.
    So \eqref{lem:sod-vb} holds as well.

    Assume \eqref{lem:sod-v}.
    If $a \rightarrowtail x \twoheadrightarrow b$ is an exact sequence with $a \in \cA$ and $b \in \cB$, then we obtain for all $a'\in \cA$ a fibre sequence
    \[ \Hom_\cE(a',a) \to \Hom_\cE(a',x) \to \Hom_\cE(a',b) \simeq 0, \]
    so $a \in \cA$ is a right adjoint object to $x$.
    Similarly, $b \in \cB$ is a left adjoint object of $x$, so we obtain adjoint functors $p \colon \cE \to \cA$ and $q \colon \cE \to \cB$. 
    Since $p$ preserves projections by assumption,
    it follows that $p$ is an exact functor.
    To see that $q$ is exact, it suffices to show that $q$ preserves inclusions.
    For any exact sequence $x \rightarrowtail y \twoheadrightarrow z$, we obtain a commutative diagram
    \[\begin{tikzcd}
        p(x)\ar[r, tail]\ar[d, tail] & x\ar[r, two heads]\ar[d, tail] & q(x)\ar[d] \\
        p(y)\ar[r, tail] & y\ar[r, two heads] & q(y)
    \end{tikzcd}\]
    The induced morphism $p(y) \cup_{p(x)} x \to y$ is an inclusion by \cref{lem:five-lemma} because $p(z) \rightarrowtail z$ is an inclusion.
    Applying \cref{lem:five-lemma} again, we conclude that $q$ preserves inclusions.
    This proves \eqref{lem:sod-vi}.

    Finally, if \eqref{lem:sod-vi} holds, the functor
    \[ f \colon \cE \to E(\cA,\cE,\cB),\ x \mapsto (p(x) \rightarrowtail x \twoheadrightarrow q(x))\]
    provides an exact section to $e$. 
    Since $q$ is left adjoint to $j$, there exists a natural transformation $\tau \colon f \circ e \to \id$.
    By assumption, we obtain for every object $a \rightarrowtail x \twoheadrightarrow b$ of $E(\cA,\cE,\cB)$ a commutative diagram with exact rows and columns
    \[\begin{tikzcd}
            p(a)\ar[r, tail]\ar[d, tail] & p(x)\ar[r, two heads]\ar[d, tail] & 0\ar[d, tail] \\
            a\ar[r, tail]\ar[d, two heads] & x\ar[r, two heads]\ar[d, two heads] & b\ar[d, two heads] \\
            0\ar[r, tail] & q(x)\ar[r, two heads] & q(b)
    \end{tikzcd}\]
    In particular, the map $q(x) \to b$ determining the component of $\tau$ at $x$ is an equivalence, and so $\tau$ is an equivalence, proving \eqref{lem:sod-i}.
\end{proof}

\begin{rem}
    By taking opposite categories, one observes that condition~\ref{lem:sod}.\eqref{lem:sod-vc} can be replaced by the assumption that $q$ preserves inclusions.
\end{rem}

\begin{rem}
     If $\cE$ is stable and $\cA$ and $\cB$ are full stable subcategories, all equipped with their maximal exact structures, note that conditions~\ref{lem:sod}.\eqref{lem:sod-iib} and \ref{lem:sod}.\eqref{lem:sod-iid} are vacuous.
     This recovers one of the usual characterisations of semi-orthogonal decompositions of stable categories.
     Semi-orthogonal decompositions of stable categories are also called semi-split Verdier sequences.

   A functor $F\colon\catst\to\cC$ with target a stable category sends semi-split Verdier sequences to (necessarily split) exact sequences if and only if it is an additive invariant
   (see \cite[Proposition 2.4]{hls:loc-thm-for-ktheory}).
\end{rem}

\begin{defn}
\ \begin{enumerate}
    \item A functor $F\colon\exact\to\Sp$ is said to be \textit{additive} if it sends semi-orthogonal decompositions of exact categories to split exact sequences of spectra.
    \item A functor $F\colon\exact\to\Sp$ is said to be \textit{invariant under passage to the stable envelope} if it sends the natural transformation $\id\to\stab$ to an equivalence.
\end{enumerate}
\end{defn}

Denote by $\iota\colon\catst\to\An$ the underlying $\infty$-groupoid functor. The following statement provides a universal property for the functor $\K\colon\catst\to\Sp$, first proved in \cite{bgt:univcharkt}.

\begin{theorem}[Blumberg--Gepner--Tabuada]\label[theorem]{thm:bgt}
    The natural transformation $\Sigma^\infty\iota\to\K$ of functors $\catst\to\Sp$ is initial amongst natural transformations $\Sigma^\infty\iota\to F$ with target an additive invariant $F$.
\end{theorem}
\begin{proof}
    Instead of \cite[Theorem~1.3]{bgt:univcharkt}, we use the non-necessarily idempotent-complete version of the universal property provided by \cite[Theorem~5.1]{hls:loc-thm-for-ktheory}. It says that the transformation $\iota \to \Omega^\infty\K$ is initial among group-like additive functors $\catst \to \An$ under the groupoid core.
    Observing that group-like additive functors $\catst \to \An$ are the same as additive functors $\catst \to \Sp_{\geq 0}$, we have for every additive functor $F \colon \catst \to \Sp$ equivalences
    \[ \Nat(\Sigma^\infty\iota,F) \simeq \Nat(\Sigma^\infty\iota, \tau_{\geq 0}F) \simeq \Nat(\K,\tau_{\geq 0}F) \simeq \Nat(\K,F) \]
    because taking connective covers preserves additive functors and the connective K-theory functor takes values in connective spectra.
\end{proof}

The functor $\K\colon\exact\to\Sp$ has the following universal property.

\begin{theorem}\label[theorem]{thm:univ-prop}
    The natural transformation $\Sigma^\infty\iota\to\K$ of functors $\exact\to\Sp$ is initial amongst natural transformations $\Sigma^\infty\iota\to F$ with target a functor $F$ which is additive and invariant under passage to the stable envelope.
\end{theorem}
\begin{proof}
    Let $F\colon\exact\to\Sp$ be an additive functor which is invariant under passage to the stable envelope.
    Write $i\colon\catst\to\exact$ for the inclusion functor, which is right adjoint to $\stab$.
    Then we have a natural equivalence $F\xrightarrow{\simeq}F\circ i\circ\stab$ which induces
    \[
        \Nat(\K, F)\simeq \Nat(\K, F\circ i\circ\stab)\simeq\Nat(\K\circ i, F\circ i)
    \]
    Since $F$ is additive, so is $F\circ i$ (but now as a functor from $\catst$) so that by \cref{thm:bgt}, $\Sigma^\infty\iota\circ i\to\K\circ i$ induces an equivalence
    \[
        \Nat(\K, F)\simeq\Nat(\Sigma^\infty\iota\circ i, F\circ i).
    \]
    Adjoining back and using the invariance under passage to the stable envelope again, we can rewrite the right hand side as $\Nat(\Sigma^\infty\iota, F)$, which concludes the proof.
\end{proof}

\begin{rem}
    Up to passing from the 1-categorical to the $\infty$-categorical world, the above theorem implies that Theorem~2 and Theorem~3 of Quillen's original article on higher K-theory \cite{quillen:higher-k} capture the properties for which the K-theory of exact categories is universal.
\end{rem}

Observe in particular that this implies that $\K$ is right Kan extended from its values on $\catst$. In fact, we can show a more general result. First, note the following.

\begin{lem} \label[lem]{lem:sod-left-special}
    Let $(\cA, \cB)$ be a semi-orthogonal decomposition of $\cE$. Then the inclusion $\cA\subseteq\cE$ is left special and the inclusion $\cB\subseteq\cE$ is right special.
\end{lem}
\begin{proof}
    Let $x\rightarrowtail y\twoheadrightarrow a$ be an exact sequence in $\cE$ with $a \in \cA$. Then $p(a) \to a$ is an equivalence, so the commutative triangle
    \[\begin{tikzcd}
        p(y)\ar[d]\ar[dr, two heads] & \\
        y\ar[r, two heads] & a
    \end{tikzcd}\]
    witnesses that $\cA$ is left special in $\cE$. By considering opposite categories, it follows that $\cB$ is right special in $\cE$.
\end{proof}

\begin{prop}
    The functor $\stab\colon\exact\to\catst$ preserves semi-orthogonal decompositions.
\end{prop}
\begin{proof}
    Given a semi-orthogonal decomposition $(\cA,\cB)$, we want to show that the pair $(\stab(\cA),\stab(\cB))$ is a semi-orthogonal decomposition of $\stab(\cE)$.
    Let $i \colon \cA \to \cE$ and $j \colon \cB \to \cE$ be the inclusion functors.
    
    Due to \cref{lem:sod-left-special}, Keller's criterion~\cref{thm:keller} and its dual version \cref{cor:keller-right-special} imply that $\stab(i)$ and $\stab(j)$ are fully faithful.

    Now, note that $\stab$ is in fact a 2-functor: for any two exact categories $\cE_1, \cE_2$, there is a functor
    \[
        \begin{tikzcd}
            \Fun^\exct(\cE_1, \cE_2)\arrow[r, "(j_{\cE_2})_*"] & \Fun^\exct(\cE_1, \stab(\cE_2))\simeq\Fun^\exct(\stab(\cE_1), \stab(\cE_2))
        \end{tikzcd}
    \]
    which recovers the functoriality of $\stab$ on arrows of $\exact$ when taking underlying $\infty$-groupoids. In particular, it follows that $\stab$ preserves adjunctions.

    Since $j \colon \cE \to \stab(\cE)$ is exact, the functor $\Fun^\exct(\cE, \cE)\to\Fun^\exct(\stab(\cE), \stab(\cE))$ is exact with respect to the pointwise exact structure.
    
    Combining these observations, we see that $(\stab(\cA),\stab(\cB))$ satisfies condition~\eqref{lem:sod-vi} of \cref{lem:sod}, so we are done.
\end{proof}

In particular, it follows from the proposition that the category $\Fun_{\add}(\catst, \Sp)$ of additive invariants defined on $\catst$ identifies as the reflexive subcategory of $\Fun_{\add}(\exact, \Sp)$ spanned by those additive invariants defined on $\exact$ which are invariant under passage to the stable envelope, or equivalently, which are right Kan extended from $\catst$.

\begin{rem}
 For the sake of completeness, let us point out that in complete analogy to the case of stable categories $\Omega^\infty\K \colon \exact \to \An$ is also the initial group-like additive invariant under the groupoid core.
 The proof of this statement presented in \cite[Section~5]{hls:loc-thm-for-ktheory} generalises in straightforward fashion from stable to exact categories after taking note that the additivity theorem holds for the connective K-theory of exact categories.
 This in turn can be seen by observing that the proof of additivity in \cite[Section~4]{hls:loc-thm-for-ktheory} also works for the Q-construction applied to exact categories.
 Alternatively, one combines the additivity theorem for Waldhausen categories proved by Barwick \cite{barwick:ktheory} with the fact that the Q-construction is the edgewise subdivision of the $S_\bullet$-construction.

 This universal property is different from the one established in \cref{thm:univ-prop} since there exist additive invariants which are not invariant under passage to the stable envelope.
 For example, the functor $\cE \mapsto \K(\Fun^\exct(\Sp^\omega,\cE))$ is an additive invariant, restricts to the ordinary K-theory functor on stable categories, but vanishes on exact $1$-categories.
\end{rem}

\bibliographystyle{alpha}
\bibliography{HeartStructures}

\end{document}